%% file: intpoints.tex
\documentclass[12pt,utf8,english,letter,nothms]{aclart}
\usepackage{smfenum,bull}
\newif\ifsmfart \smfarttrue

\usepackage{srcltx}
\usepackage[OT2,T1]{fontenc}
\usepackage{lmodern}

\usepackage{amssymb,xspace}
\usepackage[matrix,arrow]{xy}

\DeclareSymbolFont{cyrletters}{OT2}{wncyr}{m}{n}
\DeclareMathSymbol{\Sha}{\mathalpha}{cyrletters}{"58}

 \let\mathsec\mathsf
 \IfFileExists{mathrsfs.sty}
  {\usepackage{mathrsfs}\let\mathcal\mathscr}
  {\let\mathscr\mathcal}

\makeatletter \let\smf@boldmath\relax\makeatother

\advance\headheight 2pt
\calclayout
\allowdisplaybreaks[3]

\theoremstyle{plain}
\newtheorem{prop}[subsubsection]{Proposition}

\newtheorem{theo}[subsubsection]{Theorem}
\newtheorem{coro}[subsubsection]{Corollary}

\newtheorem{lemm}[subsubsection]{Lemma}

\newtheorem{defi}[subsubsection]{Definition}

\theoremstyle{definition}

\theoremstyle{remark}
\newtheorem{rema}[subsubsection]{Remark}

 \makeatletter

\def\subsection{\@startsection{subsection}{2}%
   \z@{.7\linespacing\@plus.3\linespacing}{0.3\linespacing}
   {\normalfont\bfseries\smf@boldmath}}

 \let\c@equation\c@subsubsection
 \let\cl@equation\cl@subsubsection
 \let\theequation\thesubsubsection

\def\l@table{\@tocline{0}{3pt plus2pt}{0pt}{}{\itshape}}

\makeatother

\def\card{\operatorname{Card}}
\def\Card#1{\mathopen{|}#1\mathclose{|}}

\def\Clan{\mathscr C^{\text{\upshape an}}}
\def\Clanmax{\mathscr C^{\text{\upshape an,max}}}
\def\Aff{{\mathbf A}}
\def\AD{{\mathbb A}}

\def\C{{\mathbf C}}

\def\gm{{\mathbf G}_m}

\def\Q{{\mathbf Q}}
\def\R{{\mathbf R}}

\def\Z{{\mathbf Z}}

\def\Tube{{\mathsf T}}

\let\ra\rightarrow

\let\epsilon\varepsilon \let\eps\epsilon
\let\epsilon\varepsilon
\let\phi\varphi
\let\emptyset\varnothing
\let\leq\leqslant
\let\geq\geqslant

\let\ge\geq

\def\Lie{\operatorname{Lie}}

\def\eff{{\text{\upshape eff}}}
\def\abs#1{\left\lvert{#1}\right\rvert}
\def\norm#1{\left\|{#1}\right\|}

\def\DeclareMathOperator#1#2{\def #1{\operatorname{#2}}}
 
\DeclareMathOperator{\Re}{Re} 
\DeclareMathOperator{\Im}{Im} 
\DeclareMathOperator{\Pic}{Pic}
\DeclareMathOperator{\Gal}{Gal}
\DeclareMathOperator{\Br}{Br}

\DeclareMathOperator{\Spec}{Spec}

\DeclareMathOperator{\rang}{rank}
\DeclareMathOperator{\rank}{rank}

\DeclareMathOperator{\div}{div}
\DeclareMathOperator{\Val}{Val}
\DeclareMathOperator{\vol}{vol}
\DeclareMathOperator{\Hom}{Hom}

\DeclareMathOperator{\Res}{Res}
\DeclareMathOperator{\Val}{Val}

\DeclareMathOperator{\Fr}{Fr}

\def\disc{\operatorname{disc}}
\def\Br{\text{\upshape Br}}
\def\fin{\text{\upshape fin}}
\let\bar\overline

\def\bmA{{\bar{\mathscr A}}}

\def\theenumi{\alph{enumi}}\def\labelenumi{\theenumi)}

\def\eg{\emph{e.g.}\xspace}
\def\ie{\emph{i.e.}\xspace}
	
	\def\EP{\operatorname{EP}}

\title[Integral points of bounded height]{Integral points of bounded height on toric varieties}

\author{Antoine Chambert-Loir}
\address{Universit\'e de Rennes~1, IRMAR--UMR 6625 du CNRS, Campus de Beaulieu, 35042 Rennes Cedex, France}
\address{Institut universitaire de France}
\address{Institute for Advanced Study, Einstein Drive, Princeton, NJ 08540, USA}
\email{antoine.chambert-loir@univ-rennes1.fr}

\author{Yuri Tschinkel}
\address{Courant Institute, NYU, 251 Mercer St.  New York, NY 10012, USA}
\email{tschinkel@cims.nyu.edu}

\makeatletter
 \@input{volumes.aux}
\makeatother

\begin{document}
\date{\today}
 
\begin{abstract}
We establish asymptotic formulas  for the number of integral points of bounded height
on toric varieties.
\end{abstract}
 
\ifsmfart 
 \begin{altabstract}
Nous \'etablissons un d\'eveloppement asymptotique du nombre
de points entiers de hauteur born\'ee dans les vari\'et\'es toriques.
 \end{altabstract}
\fi


\maketitle

\tableofcontents

\input{intro}
\input{torique}

\def\noop#1{\ignorespaces}

\bibliographystyle{smfplain}

\bibliography{aclab,acl}

\end{document}

%% file: intro.tex
\section{Introduction}

In this paper we study the distribution of integral points of
bounded height on toric varieties, \ie, quasi-projective algebraic varieties defined over number 
fields, equipped with an action of an algebraic torus~$T$
and containing~$T$ as an open dense orbit.

The case of projective compactifications has been the subject of intense study.
It has been treated completely over number fields via \emph{adelic harmonic analysis}
by Batyrev and the second author
in a series of papers, see~\eg, \cite{batyrev-t95b,batyrev-t96,batyrev-t98b}.
Subsequently, Salberger~\cite{salberger98} and de la Bret\`eche~\cite{breteche98b}
provided an alternative proof which relies on the parametrization of rational points
by integral points on certain descent varieties called \emph{universal torsors.}
These papers use a canonical height on toric varieties which reduces to the standard
Weil height (maximum of absolute values of coordinates) 
in the case when $X$ is a projective space.
Some other choices of heights have also been considered, at
least for projective spaces, \eg,  \cite{schanuel79,essouabri2005}.
Both methods, harmonic analysis and passage to universal torsors,
have been applied in the function 
field case by Bourqui~\cite{bourqui2003b,bourqui2011}.

These results were motivated by conjectures of Batyrev, Manin, Peyre and others
concerning the asymptotic behavior of the number of points of bounded height 
in algebraic varieties over number fields~\cite{franke-m-t89,batyrev-m90,peyre95,batyrev-t98}.
They stimulated the study of height zeta functions of equivariant compactifications
of other algebraic groups and homogeneous spaces~\cite{strauch-t99,chambert-loir-t2002,shalika-tb-t2007}, as well as the study of universal torsors over Del Pezzo surfaces.

A related, classical, problem in number theory is the study of \emph{integral points}
on algebraic varieties, 
for example 
complete intersections of low degree (circle method, \cite{birch62}),
algebraic groups or homogeneous spaces of semisimple groups (via ergodic theory
or spectral methods, \cite{duke-r-s1993,eskin-mcmullen1993,eskin-m-s1996,borovoi-r1995}).

In this paper, as well as in~\cite{chambert-loir-tschinkel2009b},
we apply the geometric and analytic framework proposed 
in~\cite{chambert-loir-tschinkel2010} to ``interpolate'' between these two counting problems.
Precisely, let $X$ be a smooth projective toric variety over a number field~$F$,
let $T$ be the underlying torus, and let $U\subseteq X$ be the complement
of a $T$-stable divisor~$D$ in~$X$.
We establish an asymptotic formula for the number of integral points of bounded height
on~$U$. The notion of integral points depends on the choice of a model of~$U$
over the ring~$\mathfrak o_F$ of integers of~$F$, while the normalization
of the height is given by the log-anticanonical divisor $-(K_X+D)$ of the pair $(X,D)$; we
assume that this log-anticanonical divisor belongs to the interior
of the effective cone of~$X$.
In Theorem~\ref{thm:integral-main}, we prove that 
\[ N(B) /  B (\log B)^{b-1} \]
has a limit~$\Theta$, when $B$ grows to infinity,
where $b$ is some explicit positive integer described below.

We first need to recall two definitions
from~\cite{chambert-loir-tschinkel2010}.
First (Definition~2.2), $\EP(U_{\bar F})$ is the virtual $\Gal(\bar F/F)$-module
given
by 
\[ \left[\mathrm H^0(U_{\bar F},\gm)/\bar F^*\right]_\Q
 - \left[ \mathrm H^1(U_{\bar F},\gm)\right]_\Q  \]
and $r(\EP(U))$, the dimension of the subspace of invariants under~$\Gal(\bar F/F)$,
is given by
\[ r(\EP(U)) = \rank (\mathrm H^0(U,\gm)/F^*) - \rank (\Pic(U)). \]
Secondly (\S3.1), 
for any place~$v$ of~$F$,
the analytic Clemens complex~$\Clan_v(D)$, 
is a simplicial complex which
encodes the incidence properties of the $v$-adic manifolds 
given by the irreducible
components of~$D$.
In this language, the integer~$b$ is given by
\[ b= r(\EP(U))+\sum_{v\mid\infty} (1+\dim\Clan_v(D)); \]

We also give a formula for the limit~$\Theta$ 
(see Theorem~\ref{thm:integral-main}). It involves 
the following analytic and geometric constants:
\begin{itemize}
\item volumes of adelic subsets with respect to suitable Tamagawa measures;
\item local volumes (at archimedean places) of minimal strata of boundary  components
of~$D$;
\item characteristic functions of certain variants of the effective cones of~$X$
attached to these strata and to the Picard group of~$U$;
\item orders of Galois cohomology groups.
\end{itemize}
This explicit formula shows that the constant $\Theta$
is positive provided that there are integral points in~$\mathscr U$.

We also establish an equidistribution
theorem for integral points of~$U$ of bounded height. This is
already new for $U=X$ where we obtain that 
rational points of bounded height in~$T(F)$ equidistribute
to Peyre's Tamagawa measure on~$X(\AD_F)^{\Br(X)}$,
the subset of~$X(\AD_F)$ where the Brauer--Manin obstruction vanishes.
This refines the classical result that rational points are dense in this subset.

In a series of papers, \cite{moroz1997a,moroz1997b,moroz1999,moroz2005},
Moroz proved similar, though less precise, results for certain affine toric varieties
over $\Q$.

\medskip
Here is the roadmap of the paper. 
In Section~\ref{sec:recall}, we recall basic facts  concerning algebraic tori,
toric varieties, heights, and Tamagawa measures.
The proof of Theorem~\ref{thm:integral-main} 
is presented in Section~\ref{sec:main}.
It relies on the Poisson summation formula on the adelic torus attached to~$T$,
and follows the strategy  of~\cite{batyrev-t95b}. As was already the
case for equivariant compactifications of additive groups in~\cite{chambert-loir-tschinkel2009b},
new technical complications arise from the presence of poles of the local Fourier transforms
at archimedean places, which contribute to the main term in the asymptotic formula.

\bigskip

\emph{Acknowledgments. ---}
The first author was supported by the Institut
universitaire de France, as well as by the National Science
Foundation under agreement No.~DMS-0635607. He would also
like to thank the Institute for Advanced Study in Princeton
for its warm hospitality which permitted the completion of this paper.
The second author was partially supported by NSF grants DMS-0739380 and 
0901777.
We thank Jean-Louis Colliot-Thélène, David Harari, and  Emmanuel Peyre
for their interest in this work, and for helpful remarks.

%% file: torique.tex
\section{Toolbox}
\label{sec:recall}

We now recall basic facts concerning
algebraic tori, toric varieties, heights, and measures.

\subsection{Algebraic numbers}\label{sec.alg-numbers}
Let $F$ be a number field.
Let $\Val(F)$ be the set of normalized absolute values of~$F$.
For $v\in \Val(F)$, we write $\abs\cdot_v$ for the corresponding
absolute value, $F_v$ for the completion of~$F$ at~$v$. If $v$ is ultrametric, we also put $\mathfrak o_v$, $\varpi_v$, $k_v$ for the ring of integers, a chosen local uniformizing element and the residue field at~$v$, respectively;
we write $p_v$ for the characteristic of
the field~$k_v$ and $q_v$ for its cardinality.

We normalize the Haar measure of a local field~$E$ as in~\cite{tate67b} (p.~310)
so that the unit ball has measure
\begin{itemize}
\item 2 if $E=\R$;
\item $2\pi$ if~$E=\C$;
\item $\abs{\disc(E/\Q_p)}^{-1/2}$ if $E$ is a finite extension of~$\Q_p$.
\end{itemize}
We also define a real number~$c_E$ by the following
formula:
\begin{itemize}
\item $c_\R=2$;
\item $c_\C=2\pi$;
\item $c_E=\abs{\disc(E/\Q_p)}^{-1/2} (1-q^{-1})/\log(q))$
if $E$ is a finite extension of~$\Q_p$ and $q$ is the norm of a uniformizer.
\end{itemize}

The ring of adeles of~$F$ is the subspace~$\AD_F$
of the product ring $\prod_{v\in \Val(F)}F_v$ consisting
of families $(x_v)$ such that $x_v\in\mathfrak o_v$ for all but finitely
many places~$v$. If $S$ is a finite set of places of~$F$,
we write $\AD_F^S$ for the similar subspace obtained
by removing places in~$S$. When $S$ is the set of archimedean
places of~$F$, we write $\AD_{F,\fin}=\AD_F^S$ (finite adeles).

Let $E$ be a finite Galois extension of~$F$ and let $\Gamma=\Gal(E/F)$
be its Galois group.
For any $v\in\Val(F)$, 
we fix a decomposition group~$\Gamma_v\subset\Gamma$ at $v$
and write $\Gamma^0_v$ for its inertia subgroup.
If $v$ is finite, we fix a geometric Frobenius element~$\Fr_v\in\Gamma_v/\Gamma^0_v$.

Let $R$ be a ring, let $\bar M$ be an $R[\Gamma]$-module 
which is free of finite rank as an $R$-module.
The Artin L-function of~$\bar M$ is defined as the Euler product
\[ \mathrm L(s,\bar M)= \prod_{v\nmid \infty} \mathrm L_v(s,\bar M),
\qquad
    \mathrm L_v(s,\bar M)=\det\big(1-q_v^{-s}\Fr_v\big|\bar M^{\Gamma_v^0}\big)^{-1}.
\]

\subsection{Algebraic tori}
Let $T$ be an algebraic torus of dimension~$d$ over $F$, \ie, 
an algebraic $F$-group scheme which becomes isomorphic
to~$\mathbf G_m^d$ over an extension of~$F$.
There exists a finite Galois extension~$E$ of~$F$ 
such that $T_E\simeq\gm^d$; we fix such an extension and let
$\Gamma$ be its Galois group.

A character of~$T$ is a morphism of algebraic groups $T\ra\gm$;
and a cocharacter a morphism of algebraic groups $\gm\ra T$.
Let $\bar M=\mathrm X^*(T_{E})$ 
be the group of $E$-rational characters
of $T$, it is a torsion-free $\Z$-module of rank~$d$
endowed with an action of $\Gamma$.
The group~$\bar N$ dual to~$\bar M$ is the group of cocharacters 
of~$T_{E}$.

The group $M=\bar M^\Gamma$ is the group
of $F$-rational characters. The group $\bar N^\Gamma$
of $F$-rational cocharacters maps naturally into
the space of coinvariants $N=\bar N_\Gamma$ which
identifies with the dual of~$M$.
The map $\bar N^\Gamma\ra N$ is not an isomorphism in general.

For any place $v\in\Val(F)$, we put
\[ 
M_v= \begin{cases} \bar M^{\Gamma_v}           & \text{ for $v\nmid \infty$ } \\
                   \bar M^{\Gamma_v}\otimes \R & \text{ for $v\mid\infty$}\end{cases},
\]   
and define~$N_v$ similarly. For $v\mid\infty$, 
the perfect duality between~$\bar M$
and~$\bar N$ induces a perfect duality $M_v\times N_v\ra\R$.
If $v$ is nonarchimedean, there is a natural bilinear map $M_v\times N_v\ra\Z$
which, however, is not a perfect pairing in general.

For any nonarchimedean $v\in\Val(F)$, the bilinear
map 
\[ 
T(F_v)\times M_v \ra\Z, \qquad 
(t,m)\mapsto -\log(\abs{m(t)})/\log (q_v)\]
induces a homomorphism
$\log_v\colon T(F_v)\ra N_v$
whose  kernel~$K_v$ is the maximal compact subgroup of~$T(F_v)$
and whose image has finite index. Moreover, $\log_v$
is surjective for all ultrametric places~$v$ which are unramified in
the splitting field~$E$ (see, \eg, \cite[p.~449]{draxl1971}).
Similarly, for any archimedean~$v\in \Val(F)$,  the bilinear map
\[
T(F_v)\times \bar M^{\Gamma_v}\ra\R, \qquad 
(t,m)\mapsto \log(\abs{m(t)})\]
induces a surjective homomorphism $\log_v\colon T(F_v)\ra N_v$
whose kernel~$K_v$ is the maximal compact subgroup of~$T(F_v)$.

\subsection{Description of the adelic group}
\label{sect:adeles}
Let $\AD_F$ be the ring of adeles of~$F$. The bilinear map
\[ T(\AD_F)\times M_\R \ra \R, \qquad ((t_v),m)\mapsto \sum_{v\in\Val(F)} 
\log(|m(t_v)|)
\]
induces a surjective continuous morphism 
$ T(\AD_F)\ra N_\R$. 
This morphism admits a section, \eg,
given by $n\mapsto (t_v(n))$, where $t_v(n)=1$ if $v$ is finite,
$t_v(n)=\exp(n/[F:\Q])$ if $v$ is real, and $t_v(n)=\exp(2n/[F:\Q])$
if $v$ is complex.

Let $T(\AD_F)^1$ be its kernel.
By the product formula, $T(F)$ embeds as a discrete subgroup into~$T(\AD_F)^1$;
moreover, the quotient $T(\AD_F)^1/T(F)$ is compact.
 This induces a decomposition
\[ T(\AD_F)/T(F) \simeq N_\R \times (T(\AD_F)^1/T(F)) \]
of $T(\AD_F)/T(F)$.
The group $K_T=\prod_{v\in \Val(F)} K_v$ is the maximal compact
subgroup of~$T(\AD_F)$; it is contained in $T(\AD_F)^1$.

Let $S\subset \Val(F)$ be a finite subset containing 
the archimedean places.
The map 
\[ T(\AD_F) \ra \prod_{v\in S} T(F_v) \ra \prod_{v\in S} N_v \]
induces an isomorphism
\[ T(\AD_F)/T(F)K_T \simeq  \prod_{v\in S} N_v / T(\mathfrak o_{F,S}), \]
where $T(\mathfrak o_{F,S})=T(F)\cap \bigcap_{v\not\in S}K_v$.
The map $T(\mathfrak o_{F,S})\ra \prod_{v\in S} N_v$ has finite kernel.
Its image is a cocompact lattice in the subspace of $\prod_{v\in S}N_v$
consisting of tuples $(n_v)_{v\in S}$ such that $\sum _{v\in S}\langle m,n_v\rangle=0$
for any $m\in M$.

\subsection{Characters}
Recall that the characters of a topological group~$G$ are the 
continuous homomorphisms
to the group~$\mathbf S^1$ of complex numbers of absolute value~$1$.
They form a topological group~$G^*$.

A character~$\chi$ of $T(\AD_F)$ is the product $(\chi_v)$
of its local components:
for any~$v\in\Val(F)$, $\chi_v$ is a character of $T(F_v)$.
A local character $\chi_v$ is called unramified if it is trivial on~$K_v$;
then there exists a unique element $m(\chi_v)\in M_v$
such that 
\[
\chi_v(t)=\exp(i\langle m(\chi_v),\log_v(t)\rangle),\qquad  \text{ for all } \, t\in T(F_v).
\]
A global character~$\chi$ is called unramified if all of its
local components  are unramified, equivalently if it is trivial
on~$K_T$;
it is called automorphic if it is trivial on~$T(F)$.

The description of $T(\AD_F)/T(F)K_T$ in Section~\ref{sect:adeles}
identifies an automorphic unramified character
of~$T(\AD_F)$ as a character of $\prod_{v\in S}N_v/T(\mathfrak o_{F,S})$.
Then, $(T(\AD_F)/T(F)K_T)^*$ is the product of
the continuous group~$M_\R$ 
and the dual $\Hom(T(\mathfrak o_{F,S}),\Z)$
of the discrete group~$T(\mathfrak o_{F,S})/\text{torsion}$.

\subsection{Toric varieties}\label{sec.toric-var}
Let $X$ be a smooth projective equivariant compactification of $T$,
\ie, a smooth projective variety~$X$ over~$F$ endowed with
an action of~$T$, and containing~$T$ as a dense open orbit.
The boundary divisor is the complementary closed subset $X\setminus T$;
it is the opposite~$-K_X$ of a canonical divisor.

By the general theory of toric varieties over algebraically closed
fields, we may assume, extending~$E$ if necessary,
that the irreducible components
of the boundary divisor $X_{E}\setminus T_{E}$ are smooth and 
geometrically irreducible, and that they meet transversally.
Let $\bar{\mathscr A}$ be the set of these irreducible 
boundary components. Since $X_{E}\setminus T_{E}$
is defined over~$F$, the set $\bar{\mathscr A}$ 
admits a natural action of the Galois group~$\Gamma$, as well as of its
subgroups~$\Gamma_v$, for $v\in\Val(F)$.
We write $\mathscr A$, resp. $\mathscr A_v$ for the sets of orbits; 
the corresponding elements label $F$-irreducible, 
respectively $F_v$-irreducible, boundary components of $X\setminus T$.

For any $\alpha\in\bar{\mathscr A}$, we write $F_\alpha$
for the subfield of~$E$ fixed by the stabilizer
of~$D_\alpha$ in~$\Gamma$, and $\Delta_\alpha$ for the
sum of all irreducible components~$D_{\alpha'}$, for
$\alpha'\in\Gamma\alpha$. If $\alpha$ and~$\alpha'$ belong
to the same orbit, the fields~$F_{\alpha}$ and~$F_{\alpha'}$ are
conjugate. For any finite place~$v\in\Val(F)$, the
choice of a decomposition subgroup~$\Gamma_v$ induces
a specific place of~$F_\alpha$, still denoted~$v$, and we write~$f_\alpha$
for the degree of~$F_{\alpha,v}$ over~$F_v$.

The closed cone of effective divisors 
$\Lambda_{\rm eff}(X_{E})\subset \Pic(X_{E})_{\R}$
on~$X_{E}$
is spanned by the classes of boundary components~$D_\alpha$,
for $\alpha\in\bar{\mathscr A}$.
Similarly,
the closed cone of effective divisors 
$\Lambda_{\rm eff}(X)\subset \Pic(X)_{\R}$
on~$X$ is spanned by the classes of the divisors~$\Delta_\alpha$.
 
Viewing a character of $T_{E}$ as a rational function on $X_{E}$
and taking its divisor
defines a canonical exact sequence of torsion-free $\Gamma$-modules
\begin{equation}
\label{eqn:toric-seq}
0\ra \bar M \ra \Pic^T(X_{E}) \stackrel{\pi}{\longrightarrow} \Pic(X_{E}) \ra 0,
\end{equation}
where  $\Pic^T(X_{E})\simeq \Z^{\bar{\mathcal A}}$ 
is the group of equivalence classes of
$T_{E}$-linearized line bundles on $X_{E}$. 
(Linearized line bundles are in  canonical correspondence
with $T_{E}$-invariant divisors in~$X_{E}$, that is, 
linear combinations of boundary components.)
The injectivity on the left follows from the fact that $X_{E}$ 
is normal and projective: if a character of~$T_{E}$ 
has neither zeroes nor poles,
then it is a regular invertible function on~$X_{E}$, hence a constant.
Taking Galois cohomology and using the fact that $\Pic^T(X_{E})$
is a permutation module, we obtain the following exact sequences:
\begin{gather}
\label{eqn:m}
0 \ra M \ra \Pic^T(X) \ra \Pic(X) \ra \mathrm H^1(\Gamma,M) \ra 0 \\
\label{eqn:mv}
0 \ra M_v \ra \Pic^T(X_{F_v}) \ra \Pic(X_{F_v}) \ra \mathrm H^1(\Gamma_v,\bar M)\ra 0.
\end{gather}
Moreover, the isomorphism $\Pic^T(X_{E})\simeq \Z^{\bmA}$
induces similar isomorphisms
\[ \Pic^T(X)\simeq \Z^{\mathscr A}, \quad \Pic^T(X_{F_v})\simeq\Z^{\mathscr A_v}. \]

By duality, the map $\bar M\ra\mathbf Z^{\bmA}$ gives rise
to a morphism of tori $\prod_{\alpha\in \mathscr A} T_\alpha\ra T$,
where, for $\alpha\in\mathscr A$, $T_\alpha=\Res_{F_\alpha/F}(\gm)$
is the Weil restriction of scalars 
from~$F_\alpha$ to~$F$
of the multiplicative group.
Using this morphism, any automorphic character $\chi\in (T(\AD_F)/T(F))^*$
induces an automorphic character
of $T_\alpha(\AD_F)$, \ie, a Hecke character~$\chi_\alpha$ of~$F_\alpha$.

\subsection{Quasi-projective toric varieties}

Let $D$ be a reduced divisor in~$X$ disjoint from~$T$.
The open set $U=X\setminus D$ is then a toric variety,
non-projective for $D\neq\emptyset$.
Let $\bar{\mathscr A}_D\subset\bmA$ be the set of irreducible components
of~$D_{E}$ and $\bar{\mathscr A}_U=\bmA\setminus \bar{\mathscr A}_D$
be the complementary subset.
The irreducible components of the divisor $U_{E}\setminus T_{E}$ 
are indexed by the traces on~$U_{E}$ of the $D_\alpha$, 
for $\alpha\in\bmA_U$. The sets $\bmA_D$ and $\bmA_U$ are stable
under the action of~$\Gamma$; we let $\mathscr A_D$ and $\mathscr A_U$
be the sets of $\Gamma$-orbits (these are subsets of~$\mathscr A$).
There is a similar $\Gamma$-equivariant exact sequence
\begin{equation}
\label{eqn:toric-seq.U}
0 \ra \mathrm H^0(U_{E},\mathscr O_U^\times)/E^\times \ra \bar M \ra \Z^{\bar{\mathcal A}\setminus\bmA_D}\stackrel{\pi}{\longrightarrow} \Pic(U_{E}) \ra 0.
\end{equation}

Let $\rho=(\rho_\alpha)$ with $\rho_\alpha=0$ if $\alpha\in\mathscr A_D$
and $1$ otherwise. 
Throughout we shall assume that $\rho\in \Lambda_{\rm eff}(X)^{\circ}$, 
\ie, is contained in the interior of the image under 
$\pi$ of the simplicial cone $\R_{\ge 0}^{\mathscr A}$.  
In more geometric terms, this means that the line bundle
$-(K_X+D)$ on~$X$ is big; this includes the particular case
where $(X,D)$ is log-Fano, \ie, $-(K_X+D)$ is ample.

\subsection{Metrized line bundles and heights}
Each boundary divisor $D_\alpha$, $\alpha \in \bar{\mathscr A}$,
defines a $T_E$-linearized line bundle on $X_E$. 
We fix smooth adelic metrics on these line bundles:
by definition these are collections of metrics, at all places~$w$ of~$E$,
almost all of which come from a model of~$X_E$ defined over
the ring of integers of~$E$; the smoothness condition means
locally constant at finite places, 
and $\mathscr C^\infty$ at archimedean places.
We assume that these metrics are invariant under
the natural action the local Galois groups~$\Gamma_v$.
We also assume that the metrics on a $T$-linearized line bundle
only depend on the isomorphy class of the underlying line bundle.

For each~$\alpha\in\bar{\mathscr A}$, let $\mathsec f_\alpha$
be the canonical section of the line bundle $\mathscr O(D_\alpha)$
with divisor~$D_\alpha$.
Then the resulting height pairing  is defined by
\[ 
H :    T(\AD_E)\times \Pic^T(X_E)_{\C} \ra \C^*,
\quad
 ((x_w); \sum s_\alpha D_\alpha)) \mapsto 
        \prod_{\alpha\in\bar{\mathscr A}} \prod_{w\in\Val(E)} \norm{\mathsec f_\alpha(x_w)}^{s_\alpha/[E:F]}. \]
It is $\Gamma$-equivariant, smooth in 
the first variable and linear in the second variable.
(If $\mathbf s\in\C^{\bmA}$, we simply write 
$H(x;\mathbf s)$ for $H(x;\sum s_\alpha D_\alpha)$.)

We shall also restrict the height pairing to line bundles
defined over~$F$ and points in $T(\AD_F)$.
This is compatible with a  corresponding theory of adelic
metrics over~$F$.   Indeed,  
a component of $X\setminus T$ decomposes
over~$E$ as a sum of some divisors~$D_\alpha$ and this furnishes
a canonical adelic metric on every line bundle on~$X$.

\subsection{Volume forms and measures}

Our analysis of the number of points of bounded height
makes use of certain Radon measures on local analytic manifolds and
on adelic spaces. Here we recall the main definitions,
referring to~\cite{chambert-loir-tschinkel2010} 
for a  detailed account of the constructions of
these measures in a general geometric context.

Let $v$ be a place  of~$F$. 
We fix a Haar measure on each completion~$F_v$
of~$F$, in such a way that $\mu_v(\mathfrak o_v)=1$
for almost all finite places~$v$.
Recall that the divisor on~$X$ 
of the invariant $n$-form $\mathrm dx$ on~$T$ 
(which is well-defined up to sign)  is given by 
\[ 
\div(\mathrm dx) = -\sum_{\alpha \in {\mathscr A}} D_\alpha.\]
We now define several measures on $X(F_v)$.  
The first is a Haar measure for the torus $T(F_v)$.
It is defined ``\`a la Weil'' by 
\[
\mu'_{T,v}=\abs{\mathrm dx}_v,\]
considering the invariant form $\mathrm dx$ as a gauge form on~$T$.
Let $\tau_v(T)=\mu'_{T,v}(K_v)$ be the measure of the maximal compact
subgroup~$K_v$ of~$T(F_v)$. 
If $v$ is unramified in~$E$, then $T$ has good reduction at~$v$ and
\[ \tau_v(T)= \# T(\mathfrak o_v)/q_v^{\dim T}
=  \mathrm L_v(1, \bar M)^ {-1} \]
(see~\cite{ono1961}, 3.3),
where we extend $T$ as a torus group scheme over~$\Spec(\mathfrak o_v)$.
We shall use the renormalized measure 
\[
\mu_{T,v}=\tau_v(T)^{-1}\abs{\mathrm dx}_v.
\]
The local Peyre-Tamagawa measure on $X(F_v)$ is defined by 
\[
\mu'_{X,v}=\abs{dx}_v/\norm{dx}_v.
\]
Since $\Pic(X_{E})$ is a free~$\Z$-module of
finite rank, two other normalizations are possible:
\begin{gather*}
 \mu_{X,v} = \mathrm L_v(1,\Pic(X_{E}))^{-1} \mu'_{X,v} , \\
 \mu_{U,v}=\mathrm L_v(1,\EP(U_{\bar F}))\mu'_{X,v} ,
\end{gather*}
where $\EP(U_{\bar F})$ is the virtual Galois module
\[ \left[\mathrm H^0(U_E,\mathscr O_X^*)/E^* \right] - \left[\mathrm H^1(U_E,\mathscr O_X^*)/\text{torsion}\right]. \]

\medskip

With these normalizations, the products
of local measures converge and define measures on suitable adelic spaces:
$\prod_v \mu_{T,v}$ is a Haar measure on $T(\AD_F)$,
$\prod_v \mu_{X,v}$ and $\prod_v\mu_{U,v}$ are Radon measures
on $X(\AD_F)$ and $U(\AD_F)$, respectively
(\cite{chambert-loir-tschinkel2010}, Theorem~\ref{prop.convergence}).
For any finite $S\subset\Val(F)$, we define Radon measures
on $T(\AD_F^S)$,  $X(\AD_{F}^S)$, and $U(\AD_F^S)$, by
\begin{align*}
 \mu_T & = \mathrm L_*^S(1,\bar M)^{-1} \prod_{v\not\in S}\mu_{T,v}, \\
 \mu_X & = \mathrm L_*^S(1,\Pic (X_{E})) \prod_{v\not\in S}\mu_{X,v}, \\
 \mu_U & = \mathrm L_*^S(1,\EP(U_{\bar F}))^{-1} \prod_{v\not\in S}\mu_{U,v},\end{align*}
where $\mathrm L_*^S(1,\cdot)$ denotes the principal value
of the Artin L-function at 1, with the finite Euler factors in~$S$
removed.

\medskip

In \cite{chambert-loir-tschinkel2010},
we have also introduced \emph{residue measures}
which are Radon measures on $X(F_v)$ with support on~$D(F_v)$.
Recall that the $F_v$-analytic Clemens complex 
$\Clan_v(D)$ is a poset whose faces are pairs $(A,Z)$ where $A$ is a nonempty
subset of~$\mathscr A_v$ and $Z$ 
is an $F_v$-irreducible component of~$D_A=\bigcap_{\alpha\in A}D_{\alpha}$
such that $Z(F_v)\neq\emptyset$.
Its order relation is the one opposite to inclusion.
(In the sequel, we shall often omit the irreducible component~$Z$
from the notation.)

For each $\alpha\in \mathscr A_v$,
we let $\Aff_{F_{\alpha,v}}$ be the Weil restriction of scalars
of the affine line from~$F_{\alpha,v}$ to~$F_v$; it is an affine space
of dimension~$[F_{\alpha,v}:F_\alpha]=\Card{\alpha}$. The norm
map~$\mathrm N\colon F_{\alpha,v}\ra F_v$ induces a polynomial
function~$\mathrm N$ on~$\Aff_{F_{\alpha,v}}$ which defines the origin
on the level of~$F_v$-rational points. By abuse of notation,
we write $\mathrm dx_\alpha$ for the top differential form
on~$\Aff_{F_{\alpha,v}}$ deduced from the one-form~$\mathrm d x$ on~$\Aff^1$.

Let $x\in X(F_v)$ and let $A_x$ 
be the set of $\alpha\in\mathscr A_v$
such that $x\in D_\alpha$. 
There exists a neighborhood~$U_x$ of~$x$ in~$X(F_v)$ and a smooth map 
$(x_\alpha)_{\alpha\in A}\colon U_x\ra \prod_{\alpha\in A}\Aff_{F_{\alpha,v}}$
which defines $D_A(F_v)$ in a neighborhood of~$x$.

Fix a pair~$(A,Z)$ in~$\Clan_v(D)$. 
The description above shows that $Z(F_v)$
is a smooth $v$-adic submanifold of $X(F_v)$ of
codimension~$\sum_{\alpha\in\mathscr A_v}\Card{\alpha}$.
Moreover, its canonical bundle admits a metric, defined inductively
via the adjunction formula, in such a way that for
any local top differential form~$\omega$ on~$Z(F_v)$,
\[ \norm{\omega} = \norm{\tilde\omega \wedge \bigwedge_{\alpha\in A}\mathrm dx_\alpha} 
\prod_{\alpha\in \mathscr A_v} \lim_x \frac{\norm{\mathsec f_{\alpha}}}{\abs{\mathrm N(x_\alpha)}}, \]
where $\tilde\omega$ is any local differential form on~$X(F_v)$ which
extends~$\omega$.
This gives rise to a measure~$\tau_{(A,Z)}$ on~$Z(F_v)$.
As in~\cite{chambert-loir-tschinkel2010}, we normalize this
measure further, multiplying it by the finite 
product~$\prod_{\alpha\in A}c_{F_\alpha}$ of constants
defined as in Section~\ref{sec.alg-numbers}.

\section{Integral points}
\label{sec:main}

\subsection{Setup}


Let $F$ be a number field, $T$ an algebraic torus defined over~$F$,
and $X$ a smooth projective equivariant compactification of~$T$.
Let $D$ be a reduced effective 
divisor in~$X\setminus T$ and let $U=X\setminus D$.
We assume that the divisor $-(K_X+D)$ on~$X$ is big.

Let $\mathscr U$ be a fixed flat $\mathfrak o_F$-scheme of finite type
with generic fiber~$U$. A rational point $x\in T(F)$
will be called $\mathfrak o_F$-integral if there exists a
section $\eps_x\colon\Spec(\mathfrak o_F)\ra \mathscr U$
which extends~$x$. Similarly, for any finite place~$v\in\Val(F)$,
a point $x\in T(F_v)$ will be called $\mathfrak o_v$-integral if it extends
to a section $\Spec(\mathfrak o_v)\ra\mathscr U$.
For any finite place~$v$, we write
$\delta_v$ for the set-theoretic
characteristic function of the set of $\mathfrak o_v$-integral points
in~$T(F_v)$. It is a locally constant function on $T(F_v)$
whose support is relatively compact in~$U(F_v)$.
For any archimedean  place~$v$, we put $\delta_v=1$ and write
\[ 
\delta = \prod_{v\in \Val(F)} \delta_v.
\] 

The generating Dirichlet series of integral points
is called the \emph{height zeta function}; it takes the form
\[ 
\mathrm Z(\mathbf s) =
\sum_{x\in T(F)\cap  \mathscr U(\mathfrak o_{F})} H(x;\mathbf s)^{-1}
= \sum_{x\in T(F)} H(x;\mathbf s)^{-1} \delta(x).
\] 
This series converges absolutely and uniformly
when all coordinates of~$\mathbf s$ have a sufficiently large real part,
and defines a holomorphic function in that domain.
Formally, we have the Poisson formula
\[
\mathrm Z(\mathbf s) = \int \hat{H}(\chi;\mathbf s) \mathrm d\chi,
\] 
where the integral is over the locally compact abelian group of characters of 
$T(\AD_F)/T(F)$ with respect to an appropriate Haar measure $\mathrm d\chi$,
and  
\[ 
\hat{H}(\chi;\mathbf s)=\int_{T(\AD_F)} H(x;\mathbf s)^{-1}\delta(x) \chi(x) \mathrm d\mu_T(  x)
\] 
is the corresponding 
Fourier transform of the height function with respect to the fixed Haar measure $\mathrm d\mu_T$. 
As in the study of rational points in \cite{batyrev-t95b,batyrev-t96,batyrev-t98b},
we investigate the height
zeta function by proving first that the Poisson formula applies;
its right-hand-side provides
a meromorphic continuation for the height zeta function.
A Tauberian theorem will then imply an asymptotic expansion for the 
number of integral points of bounded height.

If $V$ is a finite-dimensional real vector space and $\Omega$
an open subset of~$V$, we write $\Tube(\Omega)=\Omega+\mathrm iV$
for the \emph{tube domain} of~$V_\C$ over~$\Omega$.
If $V$ has explicit coordinates~$(x_\alpha)$,
we shall also write $\Tube_{>\delta}$ for the tube domain over
the open subset~$\Omega_\delta$ defined by the inequalities $x_\alpha>\delta$.

\subsection{Fourier transforms at finite places}

\begin{lemm}\label{lemma.compact.open}
For any finite place~$v$ of~$F$ and any character
$\chi_v\in T(F_v)^*$ the local Fourier transform 
$\hat H_v(\chi_v;\mathbf s)$ converges absolutely 
if $\Re(s_\alpha)>0$ for all $\alpha\not\in\mathscr A_U$ and defines
a holomorphic function of $\mathbf s$ in the tube
domain of~$\C^{\mathscr A}$ defined by these inequalities.

Moreover, there exists a compact open subgroup~$K_v$ of $T(F_v)$,
equal to the maximal compact subgroup for almost all~$v$,
such that $\hat H_v(\chi_v;\mathbf s)=0$ for any character~$\chi_v$
which is nontrivial on~$K_v$.
\end{lemm}
\begin{proof}
The first part is a special case of our results
concerning geometric Igusa functions (\cite{chambert-loir-tschinkel2010}, Section~\ref{ss.geometric-igusa}).
For the second, observe that we assumed the metrics to be locally
constant at finite places, and the same holds
for the characteristic function of the set of local integral
points. As a consequence, there exists
a compact open subgroup~$K_v$ ot~$T(F_v)$ such that the height
function $H_v(\mathbf s;\cdot)$ is $K_v$-invariant.
It follows that the Fourier transform vanishes at any character
which is not trivial on~$K_v$. Moreover, the adelic 
condition on the metrics, and the fact that
the chosen integral model of the toric varieties~$U$ and~$X$ are toric schemes
over a dense open subset of~$\Spec(\mathfrak o_F)$,
imply that for almost all~$v$, one
can take~$K_v$ to be the maximal compact subgroup of~$T(F_v)$.
\end{proof}

\begin{lemm}\label{lemma.denef}
For almost all finite places~$v$, and all $\mathbf s$
such that $\Re(s_\alpha)>0$ for all $\alpha\not\in\mathscr A_U$,
one has
\[
 \hat H_v(\mathbf 1;\mathbf s)  = \tau_v(T)^{-1} q_v^{-\dim X} \sum_{A\subset\mathcal A_U}
     \# D_A^\circ(k_v) \prod_{\alpha\in A} \frac{q_v^{f_{\alpha,v}}-1}{q_v^{f_{\alpha,v}s_\alpha}-1}.
\]
\end{lemm}
\begin{proof}
Let $\mathscr X$ be a flat projective $\mathfrak o_F$-scheme with
generic fiber~$X$ and $\mathscr D$ the Zariski closure
of~$D$ in~$\mathscr X$. There exists a finite set of places~$S$
in~$\Val(F)$ such that, after restriction to $\mathfrak o_{F,S}$,
$\mathscr X$ is a smooth toric scheme, $\mathscr X\setminus\mathscr D$ is
equal to~$\mathscr U$, and all local metrics are defined by
the given model.

For $v\not\in S$, one may compute $\hat H_v(\mathbf 1;\mathbf s)$
using Denef's formula.  Using the good reduction hypothesis,
the set of integral points in $T(F_v)$ is equal to $T(F_v)\cap\mathscr U(\mathfrak o_v)$. Moreover, $U(\mathfrak o_v)\cap (U\setminus T)(F_v)$
has measure zero with respect to the measure~$\mu_{X,v}$.
We thus can split the integral over the residue classes,
introduce local coordinates, and write
\begin{align*}
 \hat H_v(\mathbf 1;\mathbf s) & = 
   \int_{T(F_v)} H_v(x;\mathbf s)^{-1} \delta_v(x) \mathrm d\mu_{T,v}(x) \\
&= \tau_v(T)^{-1} \int_{\mathscr U(\mathfrak o_v)}
       H_v(x;\mathbf 1-\mathbf s) \mathrm d\mu_{X,v} (x)\\
& = \tau_v(T)^{-1} \sum_{A\subset\mathcal A_U}
          \# D_A^\circ(k_v) q_v^{\Card{A}-\dim(X)} \prod_{\alpha\in A} \int_{\mathfrak m_{\alpha,v}}\abs{\mathrm N_{F_{\alpha,v}/F_\alpha}(x_\alpha)}^{s_\alpha-1}  \,\mathrm dx_\alpha \\
& = \tau_v(T)^{-1} q_v^{-\dim (X)} \sum_{A\subset\mathcal A_{U}}
     \# D_A^\circ(k_v) \prod_{\alpha\in A} \frac{q_v^{f_{\alpha,v}}-1}{q_v^{f_{\alpha,v} s_\alpha}-1}.
\end{align*}
(See \cite{chambert-loir-tschinkel2010}, 4.1.6, for more details.)
\end{proof}

\subsection{The product of local Fourier transforms at finite places}

Let $K_{H,\fin}=\prod_{v\nmid\infty} K_v$ be the product of compact
open subgroups given by Lemma~\ref{lemma.compact.open}
and $T(\AD_F)^*_K\subset T(\AD_F)^* $ the subgroup of characters
whose restriction to~$K_{H,\fin}$ is trivial. Let $S$ be the finite
set of those finite places~$v$ such that either $K_v$ is distinct from
the maximal compact subgroup of~$T(F_v)$ or 
Lemma~\ref{lemma.denef} fails for $v$.

For any character $\chi\in T(\AD_F)^*_K$ and
$\mathbf s\in\C^{\mathscr A}$ such that $\Re(s_\alpha)>0$
for $\alpha\not\in\mathscr A_U$, define
\[ 
\hat{H}_{\fin }(\chi;\mathbf s) = 
\prod_{v\nmid\infty}
\hat H_v(\chi_v;\mathbf s)=
\prod_{v\nmid\infty}
\int_{T(F_v)} H_v(x_v;\mathbf s)^{-1} \delta_v(x_v) \chi_v(x_v)\mathrm d\mu_{T,v}(x_v). 
\] 
In this section, we study the convergence 
of this infinite product
and its analytic properties with respect to~$\mathbf s$ and~$\chi$.

\begin{lemm}
\label{lemm:toric-outside-S}
Let $\Omega\subset\R^{\mathscr A}$ be the open subset
of all~$\mathbf s\in\mathbf R^{\mathscr A}$
such that $s_\alpha>1/2$ for $\alpha\in\mathscr A_U$.
The infinite product~$\hat H_{\fin}(\chi;\mathbf s)$ converges
whenever $\Re(s_\alpha)>1$ for all $\alpha\in\mathscr A_U$
and extends to a meromorphic function of~$\mathbf s\in\Tube(\Omega)$.

More precisely, for each $\chi\in T(\AD_F)^*_K$,
there exists a holomorphic function $\phi(\chi;\cdot)$
on~$\Tube(\Omega)$ such that 
\[ 
\hat{H}_{\fin }(\chi;\mathbf s)  =\phi(\chi;\mathbf s)
  \prod_{\alpha \in \mathscr A_U} \mathrm L(s_{\alpha}, \chi_{\alpha}).
\]
Moreover, for any positive real number~$\eps$, there exists~$C(\eps)$
such that
$\abs{\phi(\chi;\mathbf s)}\leq C(\eps)$ for any character~$\chi$
and any $\mathbf s\in\Omega$ such that $\Re(s_\alpha)>\frac12+\eps$
for all $\alpha\in\mathscr A_U$.
\end{lemm}
(In that formula, $\chi_\alpha$ is the Hecke character
of~$F_\alpha$ deduced from~$\chi$, as in Section~\ref{sec.toric-var}.)
\begin{proof}
This is a slight modification of the proof provided in \cite{batyrev-t95b}, in
the projective case. 
For any finite place~$v\in\Val(F)$, let us define a function $\phi_v$
on~$\Tube(\Omega)$ by the formula  
\[\hat H_v(\chi_v;\mathbf s) 
= \phi_v(\chi_v;\mathbf s) \prod_{\alpha\in\mathcal A_U} \mathrm L_v(s_\alpha,\chi_\alpha), \]
where $\chi_v$ is the local component at~$v$ of a character~$\chi\in T(\AD_F)^*_K$.

For $v\not\in S$, $\chi_v$ is unramified and hence takes the form
\[ 
\chi_v(x)=H_v(x;-\mathrm im(\chi_v)),
\] 
for some $m(\chi_v)\in M_v$, where we used the 
injection~\eqref{eqn:mv} to embed $M_v$ into $\Z^{\mathscr A_v}$.
Consequently, for any such~$v$, one has
\begin{align*} \hat H_v(\chi_v;\mathbf s)  
& = \int_{T(F_v)} H_v(x;\mathbf s)^{-1}\chi_v(x)\delta_v(x)\,\mathrm d\mu_{T,v}(x)\\
& = \int_{T(F_v)} H_v(x;\mathbf s+\mathrm im(\chi_v))^{-1}\delta_v(x)\,\mathrm d\mu_{T,v}(x)\\
& = \hat H_v(\mathbf 1;\mathbf s+\mathrm im(\chi_v)).\end{align*}
Observe that
\[ \mathrm L_v(s_\alpha,\chi_\alpha)=\mathrm L_v(s_\alpha+\mathrm i m(\chi_\alpha),\mathbf 1) = \zeta_{F_\alpha,v}(s_\alpha+\mathrm i m(\chi_\alpha)).\]
Lemma~\ref{lemma.denef} implies that $\phi_v$ is
holomorphic on its domain. Moreover, for any positive
real number~$\eps$,
there is an upper bound of the form
\[ \abs{\phi_v(\chi;\mathbf s)-1 }\ll q_v^{-\min(1+2\eps,3/2)} \]
for all~$\mathbf s$ such that $\Re(s_\alpha)>\frac12+\eps$ 
for $\alpha\in\mathscr A_U$.

For $v\in S$, 
Lemma~\ref{lemma.compact.open} implies that
$\hat H_v(\chi_v;\cdot)$ is holomorphic and uniformly bounded in
this domain, independently of~$\chi$; 
the same holds for~$\phi_v(\chi_v;\cdot)$.

These estimates
imply the uniform and absolute convergence of the infinite product
$\prod \phi_v(\chi;\cdot)$ on the tube domain~$\Tube(\Omega)$;
it defines a holomorphic function $\phi(\chi;\cdot)$ on this domain.
Moreover, for any positive real number~$\eps$, there exists
a constant~$C(\eps)$ such that
\[ \abs{\phi(\chi;\mathbf s)} \ll C(\eps) \]
whenever $\Re(s_\alpha)>\frac12+\eps$ 
for $\alpha\in\mathscr A_U$.

Let $\Omega_1$ be the subset of~$\Omega$ 
defined by the inequalities~$s_\alpha>1$ for all~$\alpha\in\mathscr A_U$.
Since Hecke L-functions converge when their parameter has real part~$>1$,
the infinite product
$\hat H_{\fin }(\chi;\mathbf s)$ converges 
absolutely on~$\Tube(\Omega_1)$
and defines  a holomorphic function on this tube domain
such that 
\[ \hat H_{\fin } (\chi;\mathbf s) = \phi(\chi;\mathbf s)
 \prod_{\alpha\in\mathscr A_U} \mathrm L(s_\alpha,\chi_\alpha). \]
This provides the asserted meromorphic continuation of~$\hat H_\fin$.
\end{proof}

%
%

\subsection{Fourier transforms at archimedean places}

To establish analytic properties
of Fourier-transforms at archimedean places,
we will extend the technique
of geometric Igusa integrals developed
in~\cite{chambert-loir-tschinkel2010}.

Fix an archimedean place~$v$ of~$F$.
Since the $F$-rational divisors~$D_\alpha$  may decompose
over~$F_v$, it is natural to consider
the local height function~$H_v$ and its Fourier transform as functions
of the complex parameter~$\mathbf s\in\C^{\mathscr A_v}$.
This generalization will in fact be required
in the following sections.

Fix a splitting of the exact sequence 
\[ 0\ra T(F_v)^1\ra T(F_v)\ra M_v \ra 0 .\]
Each character of~$T(F_v)$ can now be viewed as
a pair $(\chi_1,m)$ of a character of the compact torus~$T(F_v)^1$
and an element $m\in M_v$. 
Similarly, for each $\alpha\in\mathscr A_v$, 
let $F_{\alpha,1}$ be the subgroup of $F_\alpha^\times$
consisting of elements of absolute value~$1$.
We decompose the field
$F_\alpha^\times=\R_+^\times\times F_{\alpha,1}$
and decompose the character~$\chi_\alpha$ accordingly, writing
\[
\chi_\alpha(x_\alpha)=\abs{x_\alpha}^{-\mathrm im_\alpha} \chi_{\alpha,1}(x_\alpha),
\] 
where $\chi_{\alpha,1}$ is a character of~$F_{\alpha,1}$.

\begin{lemm}
\label{lemm:toric-in-S}
Let $v$ be an archimedean place of~$F$. 
Let $\Omega_v$ be the open subset of~$\R^{\mathscr A_v}$
defined by the inequalities
$s_\alpha> -1/2$ for all $\alpha\in\mathscr A_v$,
Then, for each face~$A$
of the analytic Clemens complex~$\Clan_v(D)$, there exists
a function~$\phi_{v,A}(\mathbf s,\chi_v)$
holomorphic on the tube domain~$\Tube(\Omega_v)$
of~$\mathbf C^{\mathscr A_v}$
such that
\[
\hat{H}_v (\chi_v;\mathbf s)
=
 \sum_{A\in\Clan_v(D)} \phi_{v,A}(\chi_v;\mathbf s) 
\prod_{\alpha\in A} \frac1{s_\alpha+\mathrm im_\alpha} .
\]
Moreover, each function~$\phi_{v,A}$ is rapidly decreasing
in vertical strips; namely,
for any positive real number~$\kappa$,
one has 
\[ \abs{\phi_{v,A}(\chi_v;\mathbf s)} \ll  (1+\norm{\Im(\mathbf s)}+\norm{m(\chi_v)}+ \norm{\chi_{v,1}})^{-\kappa}, \]
provided that $\mathbf s\in\Tube(\Omega_v)$.

Assume that $\mathbf s$ and~$A$ are such that 
$s_\alpha=0$ for $\alpha\in A$. If there
exists an $\alpha\in A$ such that $\chi_\alpha$ is ramified,
then $\phi_{v,A}(\chi_v;\mathbf s)=0$.
\end{lemm}
\begin{proof}
The proof is a variant of the analysis conducted 
in our paper~\cite{chambert-loir-tschinkel2010}.
Let us consider a partition of unity $(h_A)$, indexed
by the faces~$A\subset\mathscr A_v$
of the analytic Clemens complex~$\Clan_v(D)$ at~$v$
such that the only divisors~$D_\alpha$ which intersect
the support of~$h_A$ are those with index $\alpha\in A$.
Up to refining this partition of unity, 
we also assume that on the support of~$h_A$, there
is a smooth map $(x_\alpha)$ to $\prod_{\alpha\in A}F_\alpha$
such that for each~$\alpha\in A$,
$x_\alpha=0$ is a local equation of~$D_\alpha$.
By the theory of toric varieties,
we can moreover assume that the restriction
of the map $\mathbf x\mapsto x_\alpha$  to $T(F_v)$
is an algebraic character. Considering a complement torus,
we obtain a system of
local analytic coordinates~$(\mathbf x,\mathbf y)
=(x_\alpha)_{\alpha\in A},(y_\beta)_{\beta\in B}$.
In these coordinates, the character~$\chi_v$ can be expressed as
\[ \chi_v(\mathbf x) = \prod_{\alpha\in A}\chi_\alpha(x_\alpha)
 \times \chi^A(\mathbf y), \]
where $\chi^A$ is a character of~$T(F_v)$.

After the corresponding change of variables,
the integral, localized around the stratum~$D_A(F_v)$, takes the form
\[ \mathscr I_A(\chi_v;\mathbf s)= \int \prod_{\alpha\in A} \chi_\alpha(x_\alpha) \abs{x_\alpha}^{s_\alpha-1}
 \, \theta(\mathbf s,\mathbf x,\mathbf y) \chi^A(\mathbf y)
 \mathrm d\mathbf x\mathrm d\mathbf y, \]
where $\theta$ is a smooth function with compact support around the origin
in~$\prod_{\alpha\in A}F_\alpha \times F_v^{B}$.
The local integral~$\mathscr I_A$ takes the form
\[ \mathscr I_A(\chi_v;\mathbf s) = \int \prod_{\alpha\in A} \abs{x_\alpha}^{s_\alpha+\mathrm i m_\alpha-1} 
\left( \int \theta(\mathbf s,\mathbf x,\mathbf y)\chi^A(\mathbf y) \prod_{\alpha\in A}\chi_{\alpha,1}(x_{\alpha,1}) \prod_{\alpha\in A} \mathrm dx_{\alpha,1} \mathrm d\mathbf y
\right)
\prod \mathrm d\abs{x_\alpha} .
\] 
In the inner integral, the variables $x_{\alpha,1}$ run over $F_{\alpha,1}$, \ie,
$\{\pm 1\}$ or~$\mathbf S_1$, according to whether
$F_\alpha=\R$ or $F_\alpha=\C$. In the latter case, we first
perform integration by parts to establish the rapid decay
of the inner integral with respect to the discrete part~$\chi_{\alpha,1}$
of the character~$\chi_\alpha$.
Observe also that this inner integral tends to~$0$ when $\abs{x_\alpha}\ra 0$
if the character~$\chi_{\alpha,1}$ is nontrivial,
\ie, the character~$\chi_\alpha$ is unramified.

The stated meromorphic continuation can then be established,
\eg,  iteratively integrating by parts with respect 
to the variables~$\abs{x_\alpha}$ and
writing
 \[ t^{s+\mathrm im} = \frac1{s+\mathrm im} \frac\partial{\partial t} \left(t^{s+\mathrm im}\right). \]
This gives a formula as indicated, except for the rapid decay
in~$\mathbf s$. To obtain this, it suffices to perform
integration by parts with respect to invariant vector fields
in the definition of~$\hat H_v$. The point is that for any element~$\mathfrak d$
of~$\Lie(T(F_v))$, there exists a vector field~$\mathfrak d_X$ on~$X(F_v)$ 
whose restriction to~$T(F_v)$ is invariant; moreover,
$\mathfrak d_X(H(\mathbf x;\mathbf s)) H(\mathbf x;\mathbf s)^{-1}$
is a linear form in~$\mathbf s$ times a regular function on~$X(F_v)$ 
(see~\cite{chambert-loir-t2002}, Proposition~2.2).
\end{proof}

\subsection{Integrating Fourier transforms}

We now have to integrate the Fourier transform of
the height function over the group of automorphic characters.
For the analysis, it will be necessary to first enlarge the set of variables
and then restrict to a suitable  subspace.
We thus consider a variant of the height zeta function
depending on a variable 
\[ \tilde{\mathbf s}=(\mathbf s,(\mathbf s_v)_{v\mid\infty}) \in  V_\C, \]
where $V$ is the real vector space
\[ V= \Pic^T(X)_\R \oplus \bigoplus_{v\mid\infty} \Pic^T(X_v)_\R. \]
For $\tilde{\mathbf s}\in V_\C$ such that the series converges, we set:
\[ \tilde {\mathrm Z}(\tilde{\mathbf s}) =
 \sum_{x\in T(F)\cap  \mathscr U(\mathfrak o_{F})} 
 \prod_{v\nmid \infty} H_v(x;\mathbf s)^{-1}
\times
 \prod_{v\mid\infty} H_v(x;\mathbf s_v)^{-1}. \]
Formally, we again have the Poisson formula
\[
\tilde {\mathrm Z}(\tilde{\mathbf s})
= c_T \int \hat{H}(\chi;\tilde{\mathbf s}) 
 \mathrm d\chi,
\] 
the integral over the locally compact abelian group of characters of 
$T(\AD_F)/T(F)$ with respect to a chosen Haar measure $\mathrm d\chi$
on~$(T(\AD_F)/T(F))^*$, here, for $\tilde{\mathbf s}=(\mathbf s,(\mathbf s_v)) $,
\[ 
\hat{H}(\chi;\tilde{\mathbf s})=\hat H_{\fin }(\chi;\mathbf s)
  \prod_{v\mid\infty} \hat H_v(\chi_v;\mathbf s_v)\]
is the corresponding 
Fourier transform of the height function.
The constant~$c_T$ depends on the actual choice of measures,
which we now make explicit.

Fix 
a section of the surjective homomorphism $T(\AD_F)\ra M_\R^\vee$,
whose kernel~$T(\AD_F)^1$ contains $T(F)$. If $M_\R^\vee$
is endowed with the Lebesgue measure normalized by
the lattice~$M^\vee$, this gives rise to a Haar measure
on $T(\AD_F)^1$.
The section decomposes the group of automorphic characters
as $M_\R\oplus \mathscr U_T$.
Let~$\mathscr U_T^{K}=T(\AD_F)^*_K\cap \mathscr U_T$ 
be the subgroup of~$\mathscr U_T$
consisting of characters whose restriction to
the compact open subgroup~$K_{H,\fin }$ is trivial.
By Lemma~\ref{lemma.compact.open} (see also the beginning of Section~3.3),
the Fourier transform vanishes at any character~$\chi\in\mathscr U_T$
such that $\chi\not\in\mathscr U_T^K$.
We normalize the Haar measure~$\mathrm dm$ on~$M_\R$ by the lattice~$M$
and define
a Haar measure on~$(T(\AD_F)/T(F))^*$  
as the product of~$\mathrm dm$ by the counting measure on~$\mathscr U_T$.
Provided the expression in the right hand side converges absolutely, 
one can apply the Poisson summation formula and obtain
\[ \tilde Z(\tilde{\mathbf s})
 = \frac{c_T}{(2\pi)^{\rang M}} \sum_{\chi\in\mathscr U_T^K} \int_{M_\R} \hat H(\chi;\tilde {\mathbf s}+\mathrm i m)\,\mathrm dm, \]
with
\[ c_T =  \vol(T(\AD_F)^1/T(F))^{-1} .\]

We work in the tube domain~$\Tube(\Omega_{\tilde\rho})$
of $V_\C$ over the open subset $\Omega_{\tilde\rho}$
consisting of $\tilde{\mathbf s}\in V$ satisfying
the inequalities $s_\alpha>1$ for $\alpha\in\mathscr A_U$
and $s_{v,\alpha}>0$ for $\alpha\in\mathscr A_{D,v}$.
Lemmas~\ref{lemm:toric-outside-S}, \ref{lemm:toric-in-S}, and 
the moderate growth of Hecke L-functions in vertical strips
imply that the integrand decays rapidly, hence
the validity of the Poisson formula.
For any $\chi\in\mathscr U_T^K$ we set
\[ \tilde Z(\chi;\tilde{\mathbf s}) = \frac1{(2\pi)^{\rang M}}\int_{M_\R} \hat H(\chi;\tilde {\mathbf s}+\mathrm i m)\,\mathrm dm,\]
so that
\[ \tilde Z(\tilde{\mathbf s})= c_T \sum_{\chi\in\mathscr U_T^K} \tilde Z(\chi;\tilde{\mathbf s}).\]

We first analyze individually the functions~$\tilde Z(\chi;\tilde{\mathbf s})$,
for a fixed~$\chi$.
By Lemmas~\ref{lemm:toric-in-S} and~\ref{lemm:toric-outside-S}, one can write
\begin{multline*} \hat H(\chi;\tilde {\mathbf s}+\mathrm i m) \\
 = \phi(\chi_{\fin};\mathbf s+\mathrm im) 
 \prod_{\alpha\in\mathscr A_U} \mathrm L(s_\alpha+\mathrm im_\alpha;\chi_\alpha)\prod_{v\mid\infty} \sum_{A\in\Clanmax_v(D)} \frac{ \phi_{v,A}(\chi_{v,1};\mathbf s_v+\mathrm im+\mathrm im(\chi_v))}{ \prod_{\alpha\in A_v} (s_{v,\alpha}+\mathrm im+\mathrm im(\chi_{v,\alpha}))}. \end{multline*}
Let
\[
\Clanmax_\infty(D)=\prod_{v\mid\infty}\Clanmax_v(D).
\]
For any $A=(A_v)\in\Clanmax_\infty(D)$, set
\[ \hat H_A(\chi;\tilde {\mathbf s}) 
 = 
\phi(\mathbf s;\chi_{\fin}) 
 \prod_{\alpha\in\mathscr A_U} \mathrm L(s_\alpha;\chi_\alpha)\prod_{v\mid\infty} \frac{ \phi_{v,A_v}(\mathbf s_v+\mathrm im(\chi_v); \chi_{v,1})}{ \prod_{\alpha\in A_v} (s_{v,\alpha}+\mathrm im(\chi_{v,\alpha}))} \]
so that
\[ \hat H(\chi;\tilde {\mathbf s}) =
\sum_{(A_v)\in\Clanmax_\infty(D)}
\hat H_A(\chi;\tilde {\mathbf s}).\]
In turn, this decomposition of~$\hat H$ induces a decomposition
\[ \tilde Z(\chi;\tilde{\mathbf s}) = \sum_A\in\Clanmax_\infty(D) \tilde Z_A(\chi;\tilde{\mathbf s}), \]
where
\[\tilde Z_A(\chi;\tilde{\mathbf s})
= \frac1{(2\pi)^{\rang M}}\int_{M_\R}
\hat H_A(\chi;\tilde{\mathbf s}+\mathrm im)\,\mathrm dm. \]
We first analyse these series $\tilde Z_A$ separately.

For each $A\in\Clanmax_\infty(D)$, 
the function $\hat H_A(\chi;\tilde{\mathbf s})$ is a meromorphic function
on a tube domain, with poles given by affine linear forms
whose vector parts are real and linearly independent.
Now we apply a straightforward generalization of the
integration theorem~3.1.14 from~\cite{chambert-loir-tschinkel2001b},
where we only assume that the linear forms describing
the poles are linearly independent, rather than a basis of
the dual vector space.
The convergence is guaranteed by the rapid decay of
the functions~$\phi$ and~$\phi_{v,A_v}$ in vertical strips.

Let us set 
\[
 \Pic^T(U;A)= \Pic^T(U)\oplus\bigoplus_{v\mid\infty} \Z^{A_v}.
\]
There is a natural homomorphism $M\ra \Pic^T(U;A)$
and we define $\Pic(U;A)$ as the quotient $\Pic^T(U;A)/M$.

\begin{lemm}
The abelian group $\Pic(U;A)$ is torsion-free.
\end{lemm}
\begin{proof}
Let $D\in\Pic^T(U)$ and, for any place~$v$ such that $v\mid\infty$,
let $D_v\in\Z^{A_v}$. Assume that the class of $(D,(D_v))$ modulo~$M$
is torsion in~$\Pic(U;A)$.  Let $n$ be a positive integer
such that $n(D,(D_v))$ is the image of an element~$m\in M$.
We then have $nD=m$ in $\Pic^T(U)$;
since $\Pic(U)$ is torsion-free 
(\cite{fulton1993}, p.~63),
$D$ is the image of an element of~$M$ by the natural map $M\ra\Pic^T(U)$
This allows us to assume that $D=0$. 

For any archimedean place~$v$,
let $Z_v$ be any maximal stratum of $\Clan_v(X\setminus T)$
which contains~$A_v$. Since $A_v$ is a maximal stratum of
$\Clan_v(X\setminus U)$, the irreducible components of~$(X\setminus T)_v$
which corresponds to elements of~$Z_v\setminus A_v$
have to meet~$U$. As a consequence, the divisor of the character~$\chi_m$ 
is an $n$th power in the group of $T$-linearized divisors
of the affine toric variety corresponding to~$Z_v$.
Using again~\cite{fulton1993}, $\chi_m$ is a $n$th power on
this toric variety. In particular, there exists an $m_v\in M$
such that $\chi_m=\chi_{m_v}^{n_v}$ as characters
of~$T$; in other words, $m'=m/n\in \bar M^{\Gamma_v}$.
In particular, $m'\in \bar M$ and, since $m=nm'\in \bar M^\Gamma$,
we also have $m'\in M$.
We have thus proved that $(D,(D_v))$ is the image of~$m'$
by the natural map $M\ra \Pic^T(U;A)$, as was to be shown.
\end{proof}

Note that for $A=\emptyset$, one has $\Pic^T(U;\emptyset)=\Pic^T(U)$
but $\Pic(U;\emptyset )$ is a sublattice in $\Pic(U)$
of index $\Card{\mathrm H^1(\Gamma,\bar M)}$.
This discrepancy with the natural integral structure on~$\Pic(U)$
will appear later in the definition of the constant~$\Theta$
below (compared with the definitions given in~\cite{batyrev-t95b}).

Let $V_A$ be the real vector space $ V_A= \Pic^T(U;A)_\R $,
endowed with the measure normalized by the lattice $\Pic^T(U;A)$. 
We consider $V_A$ as a \emph{quotient} of~$V$ by letting 
$r_A\colon V\ra V_A$ be the map which forgets the missing components.
Let $\Lambda_A$ be the closed simplicial cone in~$V_A$ consisting
of all vectors~$\tilde{\mathbf s}=(\mathbf s,(\mathbf s_v))$
such that $s_\alpha\geq 0$ for all $\alpha\in\mathscr A_U$,
and $s_{v,\alpha}\geq 0$ for all $v\mid\infty$ and $\alpha\in A_v$.
Pulling-back via~$r_A$ the characteristic function 
of~$\Lambda_A$ we obtain  a rational function~$\mathscr X_{\Lambda_A}$
on~$V_\C$; it is given by
\[ \mathscr X_{\Lambda_A}(\tilde{\mathbf s}) = 
  \left( \prod_{\alpha\in \mathscr A_U} s_\alpha \prod_{v\mid\infty}\prod_{\alpha\in A_v} s_{v,\alpha} \right)^{-1}. \]

Set $V'=V/M$ and let $\pi\colon V\ra V'$  be the natural projection.
By Proposition~\ref{prop.maximal-order} below,
the composition $M_\R\ra V \xrightarrow{r_A}V_A$ is injective.
Let $V'_A=V_A/M_\R$ and $\pi_A\colon V_A\ra V'_A$ be the natural projection;
one has $V'_A=\Pic(U;A)_\R$;
let us endow~$V'_A$ with the Lebesgue measure normalized by~$\Pic(U;A)$. 
There exists a unique map $r'_A\colon V'\ra V'_A$ such that
$\pi_A\circ r_A=r'_A\circ \pi$, so that $V'_A$ is a quotient of~$V'$.
We let $\Lambda'_A=\pi(\Lambda_A)$ be the image of~$\Lambda_A$ in~$V'_A$;
it is an closed cone generated
by the images of the generators of~$\Lambda_A$.
We pull-back to~$V'$ the characteristic function of the cone~$\Lambda'_A$
and obtain a rational function which is given by the integral formula
\[ \mathscr X_{\Lambda'_A} (\pi(\tilde{\mathbf s})) =
 \frac1{(2\pi)^{\rang M}}\int_{M_\R} \mathscr X_{\Lambda_A}(\tilde{\mathbf s}+\mathrm im)\,\mathrm dm.\]
(See, \eg, \cite{chambert-loir-t99b}, Proposition~3.1.9.)

We first conclude that for each $\chi\in\mathscr U_T^K$,
\[  \tilde Z_A(\chi;\tilde{\mathbf s}+\tilde\rho+\mathrm i\tilde m(\chi))=
\frac1{(2\pi)^{\rang M}}\int_{M_\R} \hat H_A(\chi;\tilde{\mathbf s}+\tilde\rho+\mathrm i\tilde m(\chi)+\mathrm im)\,\mathrm dm,
\]
so that the function 
\[ \tilde{\mathbf s}\mapsto \tilde Z_A(\chi;\tilde{\mathbf s}+\tilde\rho+\mathrm i\tilde m(\chi)) \]
is holomorphic on the tube domain
over the interior of the cone~$(r'_A)^{-1}(\Lambda'_A)$
in~$\Tube(V')$.
Moreover, there exists an open neighborhood~$\Omega_A$ of the origin
in~$V'$ such that $\tilde Z_A(\chi;\tilde{\mathbf s}+\tilde\rho+\mathrm i\tilde m(\chi))$ extends to a meromorphic function
on~$\Tube(\Omega_A+(r'_A)^{-1}(\Lambda'_A))$ whose poles are given by the
linear forms on~$V'$ corresponding to the faces of~$\overline\Lambda'_A$.
Moreover, $\tilde Z_A$ decays rapidly in vertical strips and
for any positive $\tilde{\mathbf s}\in V$,
\begin{multline}\label{eq:limit-formula-A}
 \lim_{t\ra0} t^{\dim(\Lambda_A)}\tilde Z_A(\chi; t\tilde{\mathbf s}+\tilde\rho+\mathrm i\tilde m(\chi))\\
 = \mathscr X_{\Lambda'_A}(\tilde{\mathbf s})
\phi(\chi;\rho) \prod_{\alpha\in\mathscr A_U}
\mathrm L^*(1;\chi_\alpha) 
\prod_{v\mid\infty} \phi_{v,A_v}(\chi_{v,1};\mathrm im(\chi_v)).
\end{multline}
 
We now sum these meromorphic functions $\tilde Z_A(\chi;\cdot)$
over all $\chi\in\mathscr U_T^K$.  Due to the stated
decay in vertical strips, this series converges
and defines a meromorphic function with poles
given by the translates of the cones~$\Lambda'_A$ by a
discrete subgroup consisting of (the images of) imaginary
vectors $\mathrm i\tilde m(\chi)$, for $\chi\in\mathscr U_T^K$.

\begin{lemm}\label{lemm.discrete}
Under the map 
$\chi\mapsto \pi(\tilde m(\chi))$,
the group~$\mathscr U_T^K$ is mapped to a discrete subgroup of~$V'_A$.
\end{lemm}
\begin{proof}
We first show that $\tilde m(\mathscr U_T^K)$ is discrete in~$V_A$.

Let $v$ be an archimedean place of~$F$.
We proved in~\cite{chambert-loir-tschinkel2010}, Section~5.1,
that each connected component of the analytic Clemens complex of a smooth
toric variety is simplicial. Consequently, there exists a maximal
stratum of~$\Clanmax_v(X\setminus T)$  of the form~$A_v\cup B_v$,
where $B_v$ corresponds to divisors in~$(U\setminus T)_v$. 
Let $N'_v$ and $N''_v$ be the subspaces of~$N_v$ generated
by the rays corresponding to the $F_v$-irreducible components
of~$(X\setminus T)_v$ occurring in~$A_v$, resp.~$B_v$.
One has $N_v=N'_v\oplus N''_v$.
Let us also set $N_\infty=\prod N_v$, and define similarly $N'_\infty$
and $N''_\infty$. 

Let $T'=\prod_{\alpha\in\mathscr A_D} \mathbf G_{m,F_\alpha}$ 
and $T''=\prod_{\alpha\in\mathscr A_U}\mathbf G_{m,F_\alpha}$ 
be the quasi-split tori 
corresponding to $T$-linearized divisors in~$D$, resp.\ outside~$D$.
The natural map $\bar M\ra \Pic^T(X_{E})$ 
of $\Gamma$-modules induces a surjective morphism of algebraic tori
$T'\times T''\ra T$. Let $K'$ and $K''$ be compact subgroups
of~$T'(\AD_F)$ and $T''(\AD_F)$ such that $K'\times K''$ surjects onto~$K$.
Then, the map $(T(\AD_F)/T(F)K)^*\ra M_\infty$
``decomposes'' as a product
\[ (T'(\AD_F)/T'(F)K')^*\times(T''(\AD_F)/T''(F)K'')^*
  \ra  M'_\infty \times M''_\infty. \]
This identifies the image of~$\mathscr U_T^K$ in~$V_A$ as
the intersection with $M'_\infty$ of the
lattice $\mathscr U_{T'}^{K'}\times \mathscr U_{T''}^{K''}$ in
$M'_\infty\times M''_\infty$. It is therefore discrete, as claimed.

This description also shows that the image of~$M_\R$ in~$V_A$
is \emph{orthogonal} to $\tilde m(\mathscr U_T^K)$. As a consequence,
$\pi(\tilde m(\mathscr U_T^K))$ is still discrete in~$V'_A=V_A/M_\R$.
\end{proof}

This furnishes the existence and holomorphy of~$\tilde Z$
in the tube domain over the open subset~$\Omega_{\tilde\rho}$
of $V$ formed of $\tilde{\mathbf s}\in V$ such
that $\tilde{\mathbf s}-\tilde\rho$ has positive coordinates.
(Explicitly, these conditions mean that
$s_\alpha>1$ if $\alpha\in\mathscr A_U$,
and that $s_{v,\alpha}>0$ for all places~$v\mid\infty$ and all~$\alpha$ 
in a face of the analytic Clemens complex $\Clan_v(D)$.

\subsection{Restriction to the log-anticanonical line bundle}
Let us consider the particular case of the height zeta function
with respect to the log-anticanonical line bundle. 
\emph{We assume that this line bundle belongs to the
interior of the effective cone of~$X$.} Then there exists a
$\lambda$ in the interior of the effective cone of $\Pic^T(X)$
such that $\rho\sim\lambda$ in~$\Pic(X)$. Since the height
of a rational point only depends on the isomorphy class
of the underlying line bundle, one has
\[
 Z(s\rho)=Z(\rho+(s-1)\lambda) =\tilde Z(\tilde\rho+(s-1)\tilde\lambda),
\]
where $\tilde\lambda$ is the vector~$(\lambda,(\lambda))\in V$.
Observe that all components of~$\tilde\lambda$ are positive.
It follows that $s\mapsto
Z(s\rho)$ is holomorphic for $\Re(s)>1$ and has a meromorphic
continuation to the left of~$1$.

\subsection{Poles on the boundary of the convergence domain}

We now describe the poles of 
the function $s\mapsto Z(s\rho)$ which satisfy $\Re(s)=1$.

We first claim that they lie in a finite union of arithmetic
progressions. 
Indeed, according to the summation process above,
there is a pole at $1+\mathrm i\tau$ whenever there
exists $\chi\in\mathscr U_T^K$, $A=(A_v)$ a family of faces
of maximal dimension of the analytic Clemens  complexes,
such that $r_A(\tau\tilde\lambda +\tilde m(\chi))$ belongs to a face
of the cone~$\Lambda_A$.
This means that there exists $\alpha\in \bigcup A_v$
such that $\tau=-m_v(\chi_\alpha)/\lambda_\alpha$.
The result now follows from the fact that
for each fixed~$(\alpha,v)$, the image of $\mathscr U_T^K$ 
by the map $\chi\mapsto m_v(\chi_\alpha)$ is an arithmetic progression. 

Fix such a character~$\chi\in\mathscr U_T^K$ and $\tau\in\R$.
According to the limit formula~\eqref{eq:limit-formula-A}
and the vanishing of $\phi_v$
for ramified characters stated in Lemma~\ref{lemm:toric-in-S},
the order of the pole at~$1+\mathrm i\tau$ is at most equal to 
the sum $b(\tau)$ of the following integers:
\begin{itemize} 
\item minus the rank of~$M$;
\item if $\tau=0$, the cardinality of $\mathscr A_U$;
\item for each~$v\mid\infty$, the maximal cardinality
of a face $A_v\in\Clan_v(D)$
such that there exists an unramified character~$\chi_v\in\mathscr U_T$
such that for any $\alpha\in A_v$,  $m(\chi_{v,\alpha})=-\tau$.
\end{itemize}
We set $b=b(0)$; observe that 
\begin{align} b & = -\rang M+\Card{\mathscr A_U} + \sum_{v\mid\infty}(1+\dim\Clan_v(D)) \notag
\\
& = r(\EP(U)) +  \sum_{v\mid\infty}(1+\dim\Clan_v(D)), \label{eqn.b}
\end{align}
since 
\[  \rang M - \Card{\mathscr A_U} =
\rank (\mathrm H^0(U,\gm)/F^*) - \rank (\Pic(U)) = r(\EP(U)), \]
the dimension of the $\Gamma$-invariants of~$\EP(U_{\bar F})$.
We shall prove later on, 
by computing the constant term, that the order of
the pole at $s=0$ is indeed equal to~$b$
 (see Lemma~\ref{lemm:non-vanishing}).

Recall the assumption that the log-anticanonical divisor belongs
to the interior of the effective cone; a fortiori
$T\neq U$, hence $\mathscr A_U\neq\emptyset$.
Therefore,
\begin{equation}
 b > \sum_{v\mid\infty} (\dim\Clan_v(D)+1) - \rang M
 \geq b(\tau). \end{equation}

\subsection{Characters giving rise to the pole of maximal order}

Let $A(T)^*$ be
the group of all automorphic characters $\chi\in (T(\AD_F)/T(F))^*$
such that $\chi_\alpha\equiv 1$ for all $\alpha\in\mathscr A$.
\begin{lemm}\label{lemm.wa}
The group $A(T)^*$ is finite, and canonically identifies
with the Pontryagin dual of the group $T(\AD_F)/\overline{T(F)}^w$,
quotient of~$T(\AD_F)$ by the closure of~$T(F)$ for the product topology.
\end{lemm}
Note that the product topology on $T(\AD_F)$ is coarser than the adelic topology,
so that $\overline{T(F)}^w$ is indeed a closed subgroup of $T(\AD_F)$.
\begin{proof}
Let $P$ be the quasi-split torus dual to the permutation Galois module $\Pic^T(\bar X)$;
the map $\bar M\ra \Pic^T(\bar X)$ induces a morphism $\mu\colon P\ra T$ of algebraic tori.
By definition, $A(T)^*$ is the kernel of the morphism
\[ \mu^*\colon (T(\AD_F)/T(F))^*  \ra (P(\AD_F)/P(F))^*. \]
By Pontryagin duality, $A(T)^*$ is the dual of the cokernel of the map
\[ P(\AD_F)/P(F) \ra T(\AD_F)/T(F) \]
induced by~$\mu$.
Inspection of the proof of Theorem~6 of~\cite{voskresenskii1970} then shows
this cokernel is equal to the finite group $T(\AD_F)/\overline{T(F)}^w$, hence the lemma.
\end{proof}

The quotient $T(\AD_F)/\overline{T(F)}^w$ is classically denoted $A(T)$,
it  measures the obstruction to weak approximation. 
\begin{lemm}\label{lemm.brauer}
The closure $\overline{T(F)}$ of~$T(F)$ in~$X(\AD_F)$ coincides
with $X(\AD_F)^{\Br(X)}$, 
the locus in~$X(\AD_F)$ where the Brauer--Manin obstruction vanishes. 
In particular, it is an non-empty open and closed subset of~$X(\AD_F)$.
\end{lemm}
\begin{proof}
Let us observe that $T(\AD_F)$ is dense in~$X(\AD_F)$, so that
$\overline{T(F)}^w=T(\AD_F)\cap \overline{T(F)}$.
According to~\cite{sansuc1981}, Theorem~8.12 and Corollary~9.4,
$\overline{T(F)}^w$ coincides with the locus $T(\AD_F)^{\Br(X)}$ in~$T(\AD_F)$ 
where the Brauer--Manin obstruction vanishes.
Observe that, $T(\AD_F)^{\Br(X)}=T(\AD_F)\cap X(\AD_F)^{\Br(X)}$.
Since the Brauer--Manin pairing is continuous and $\Br(X)/\Br(F)$ is finite,
$X(\AD_F)^{\Br(X)}$ is open and closed in~$X(\AD_F)$. An easy topological
argument then shows that $\overline{T(F)}=X(\AD_F)^{\Br(X)}$.
\end{proof}

A character $\chi\in\mathscr U_T^K$ contributes
to a pole of order~$b$ at $s=0$ if and only if the
following properties hold:
\begin{itemize}
\item for any $\alpha\in\mathscr A_U$, the Hecke character $\chi_\alpha$ is trivial;
\item for any $v\mid\infty$, there exists a face~$A_v$ of maximal dimension
of $\Clan_v(D)$, such that for any $\alpha\in A_v$, 
the local character $\chi_{v,\alpha}$ is trivial.
\end{itemize}

%

\begin{prop}\label{prop.maximal-order}
Let $A=(A_v)_{v|\infty }$ be a family, where for each~$v$, $A_v$
is a maximal stratum of $\Clan_v(D)$.
Let $\chi\in T(\AD_F)^*$ be any topological character
satisfying the following assumptions:
\begin{itemize}
\item For all $\alpha\in\mathscr A_U$, the adelic character
$\chi_\alpha$ is trivial;
\item For any archimedean place~$v$, the restriction of the analytic
character~$\chi_v$ to the stabilizer of the stratum~$D_{A_v}$ 
in~$T(F_v)$ is trivial.
\end{itemize}
Then all archimedean components of~$\chi$ are trivial.
\end{prop}
\begin{proof}
We need to prove that for any archimedean place~$v$, the analytic
character~$\chi_v$ is trivial. Let us fix such a place.
The description of the analytic Clemens of a smooth toric variety
done in~\cite{chambert-loir-tschinkel2010}
 implies that for each~$v$, there exists
a maximal stratum~$B_v$ of $\Clan_v(X\setminus T)$ containing~$A_v$,
and this stratum is reduced to a point~$b_v\in X(F_v)$.
Moreover, there exists a maximal split torus $T'_v$ in~$T_{F_v}$
such that the Zariski closure of~$T'_v$ in~$X_{F_v}$ contains~$b_v$.
The assumptions imply that $\chi_{v,\beta}$ is trivial
for any $\beta\in B_v$. Since the corresponding cocharacters generate
the group of cocharacters of~$T'_v$, 
the analytic character~$\chi_v$ is trivial on~$T'_v$.
Moreover, the torus $T_{F_v}/T'_v$ is anisotropic by construction,
so that $T(F_v)/T'_v(F_v)$ is compact; moreover,
$T(F_v)$ is contained in the product of~$T'(F_v)$ and
of the stabilizer in~$T(F_v)$ of the stratum~$D_{A_v}$.
It follows that $\chi_v$ is trivial, as claimed.
\end{proof}

\begin{defi}
Let $A(T,U)^*$ be the subgroup of $(T(\AD_F)/T(F))^*$ 
consisting of characters~$\chi$ such that:
\begin{itemize}
\item For all $\alpha\in\mathscr A_U$, the Hecke
character~$\chi_\alpha$ is trivial;
\item For any $v\mid\infty$, the local character~$\chi_v$ is trivial.
\end{itemize}
We also write $A(T,U,K)^*$ for the subgroup of~$A(T,U)^*$
consisting of characters  which are trivial on~$K$.
\end{defi}
According to Proposition~\ref{prop.maximal-order}, the characters which contribute
to the main pole are elements of~$A(T,U,K)^*$.
We also see that the natural map $M_\R\ra V_A$ is injective.

\begin{lemm}
The group $A(T,U,K)^*$ is finite. 
\end{lemm}
\begin{proof}
Consider the natural morphism from~$A(T,U,K)^*$ to $(K_T/K)^*$.
Its image is finite. Since the archimedean components of the
characters in~$A(T,U,K)^*$ are trivial, the kernel of this morphism
is a subgroup of $(T(\AD_{F})/T(F)T(F_\infty)K_{T})^*$,
hence is finite by Theorem~3.1.1 of~\cite{ono1961}.
\end{proof}

Let $M_U$ be the kernel of the natural map $M\ra \Z^{\mathscr A_U}$
and let $T_U$ be the corresponding torus. There is a surjective
morphism of tori~$\pi\colon T\ra T_U$ whose kernel is a torus~$T'$. 

\begin{lemm}\label{lemm.ortho.A(T,U)}
The orthogonal of $A(T,U)^*$ in~$T(\AD_F)$ is a subgroup of finite
index in~$\pi^{-1}(T_U(F)) T(F_\infty)$.
\end{lemm}
\begin{proof}
Let $T(\AD_F)^\perp$ be the orthogonal of $A(T,U)^*$.
By definition of~$A(T,U)^*$, it contains $T(F_\infty)T(F)$.

We have $T'(\AD_F)\subset T(\AD_F)^\perp$. Indeed, let $\chi\in A(T,U)^*$.
For any place~$v$ of~$F$, the local component~$\chi_v$
of~$\chi$ is trivial on~$T'(F_v)$ because the Lie algebra of~$T'$ is
orthogonal to~$M_U$, by definition of~$T'$. This implies that $\chi$
is trivial on~$T'(\AD_F)$.

Furthermore, the exact sequence of cohomology 
\[ T'(F)\ra T(F) \ra T_U(F) \ra H^1(F, T'(\bar F)) \]
implies that $\pi(T(F))$ is a subgroup of finite  index in~$T_U(F)$.
The lemma follows.
\end{proof}

Let $T(\AD_F)^{\perp}$ be the orthogonal of $A(T,U)^*$ in~$T(\AD_F)$.
Since $A(T,U,K)^*$ is the intersection of~$ A(T,U)^* $ with the orthogonal
of~$K$ in~$T(\AD_F)^*$, we have
\[ T(\AD_F)^{A(T,U,K)^*} = T(\AD_F)^{\perp} K. \]
The left hand side is an open subgroup  of finite index~$\Card{A(T,U,K)^*}$
in~$T(\AD_F)$; we endow it with the Haar measure $\Card{A(T,U,K)^*} \mu_T$.
Then, there exists a unique Haar measure~$\mu_T^\perp$ on~$T(\AD_F)^\perp$ such that
for any $K$-invariant fonction~$f$ on~$T(\AD_F)$,
\begin{equation}
\label{eq.muTperp}
  \int_{T(\AD_F)^\perp} f(x)\,\mathrm d\mu_T^\perp(x)
 = \Card{A(T,U,K)^*} \int_{T(\AD_F)^{A(T,U,K)^*}}  f(x)\,\mathrm d\mu_T(x) .\end{equation}
This measure does not depend on the choice of the compact  subgroup~$K$.

%
%

\subsection{The leading term}

By the preceding analysis, with $b$ defined as in Equation~\eqref{eqn.b},
one has
\[
\lim_{t\ra 1} (t-1)^{b} Z(t\rho) 
= 
\sum_{A\in\Clanmax_\infty(D)}   \Theta_A \]
where, for each $A\in\Clanmax_\infty(D)$,
\begin{align*} \Theta_A &= 
c_T 
\sum_{\chi\in A(T,U,K)^*} 
\mathscr X_{\Lambda'_A}(\pi(\tilde\lambda)) 
\phi(\chi;\rho)\prod_{\alpha\in\mathscr A_U}\mathrm L^*(1;\chi_\alpha)
\prod_{v\mid\infty} \phi_{v,A_v}(\chi_{v,1};\mathrm im(\chi_v)) \\
&= c_T  \,
\mathscr X_{\Lambda'_A}(\pi(\tilde\lambda))  \,
\left(\lim_{t\ra 1} 
(t-1)^{b+\rang M}
\left(
\sum_{\chi\in A(T,U,K)^*}
\hat H_A(\chi;\tilde\rho+(t-1)\tilde\lambda) \right)\right).
\end{align*}

By Fourier inversion for the finite group $A(T,U,K)^*$, we have
\[
\sum_{\chi\in A(T,U,K)^*}
\hat H_A(\chi;\tilde\rho+(t-1)\tilde\lambda) 
= \Card{A(T,U,K)^*} \int_{ T(\AD_F)^{A(T,U,K)^*}}  H(x,\tilde\rho+\tilde {\mathbf s})
\delta(x)\theta_A(x) \,\mathrm d\mu_T(x). \]
Moreover, since the functions we integrate are invariant under~$K$, it follows
from the definition~\eqref{eq.muTperp}
of the measure~$\mu_T^\perp$ on~$T(\AD_F)^\perp$ that
\[
\sum_{\chi\in A(T,U,K)^*}
\hat H_A(\chi;\tilde\rho+(t-1)\tilde\lambda) 
= \int_{ T(\AD_F)^{\perp}}  H(x,\tilde\rho+\tilde {\mathbf s})
\delta(x)\theta_A(x) \,\mathrm d\mu_T^\perp(x), \]
so that
\begin{multline} \Theta_A = 
 c_T \, \mathscr X_{\Lambda'_A}(\pi(\tilde\lambda))  \times \\
\times \left(\lim_{t\ra 1}(t-1)^{b+\rang M} \int_{T(\AD_F)^{\perp}}
  H(x;\rho+(t-1)\lambda ) \delta( x)\theta_A( x)\,
   \mathrm d\mu_T ^\perp( x)\right).
\end{multline}

\subsection{Leading term and equidistribution for rational points}

We assume in this Subsection that $U=X$, \ie, we consider 
the distribution of rational points of bounded height
in toric varieties.
In this case, all analytic Clemens complexes are empty
and $\Clanmax_\infty(D)$ is understood as
the set $\{\emptyset\}$. We have 
\[ \Pic(U;\emptyset)= \Pic(X);\]
because
of the chosen normalizations for measures, 
 and the characteristic
function $\mathscr X_{\Lambda'_\emptyset}$
is the characteristic function of the effective cone~$\Lambda_\eff(X)$
in~$\Pic(X)$
multiplied by $\Card{\mathrm H^1(\Gamma,\bar M)}$.
Moreover, $T(\AD_F)^\perp$ is equal to~$T(\AD_F)^{\Br(X)}$ 
and is endowed with the Haar measure~$\Card{A(T)^*}\mu_T$.
The height zeta function
has a single pole at $t=1$ of multiplicity
\[ b=\Card{\mathscr A}-\rang M = \rang \Pic(X), \]
and 
\[
\Theta_\emptyset:=\lim (t-1)^{b}Z(t\rho)\]
is given by
\begin{equation}\label{eq.theta.empty}
 \Theta_\emptyset=  c_T \,\Card{A(T)}\,\Card{\mathrm H^1(\Gamma,\bar M)}\,
\mathscr X_{\Lambda_\eff(X)}(\rho) \, \lim_{t\ra 1} (t-1)^{\Card{ \mathscr A}}\int_{T(\AD_F)^{\Br(X)}}
 H(x;t\rho)\,\mathrm d\mu_T( x). \end{equation}

The boundary of~$T(\AD_F)^{\Br(X)}$ in $X(\AD_F)$ is contained in a countable
union of spaces of the form $\prod_{w\neq v}X(F_w)\times (X\setminus T)(F_v)$,
hence has measure~$0$ for the Tamagawa measure~$\tau_X$.
In our paper~\cite{chambert-loir-tschinkel2010},
we have computed limits as the one appearing in~\eqref{eq.theta.empty}
(Proposition~\ref{prop.zeta.global})
and derived in Theorem~\ref{theo.volume.global}
an asymptotic formula for volumes,
as well as the equidistribution property for height balls.
This analysis implies that
for any continuous function~$\phi$ on~$X(\AD_F)$,
\[ \lim_{t\ra 1} (t-1)^{\Card{\mathscr A}} \int_{{T(\AD_F)}^{\Br(X)}}
   H( x;t\rho)\phi(x)\,\mathrm d\mu_T( x)
 = \prod_{\alpha\in\mathscr A}\frac1{\rho_\alpha} \int_{X(\AD_F)}
     \phi( x)\,\mathrm d\tau_X( x), \]
where $\tau_X$ is Peyre's Tamagawa measure on~$X(\AD_F)$.
Then, the same formula also holds
for the characteristic function of
$X(\AD_F)^{\Br(X)}$ since its boundary has measure~$0$.
Therefore,
\[\Theta_\emptyset = c_T\,\Card{A(T)} \,\Card{\mathrm H^1(\Gamma,\bar M)}\,
\mathscr X_{\Lambda_\eff(X)}(\rho)  \, \tau_X(X(\AD_F)^{\Br(X)}). \]

\begin{lemm}\label{lemm.vosk}
One has
\[ c_T\,\Card{A(T)}\,\Card{\mathrm H^1(\Gamma,\bar M)}
 = \Card{\mathrm H^1(\Gamma,\Pic(\bar X))} .\]
\end{lemm}
\begin{proof}
The measure~$\mathrm d\tau_{(X,X\setminus T)}$ is exactly the Haar
measure of $T(\AD_F)$ used by Ono in~\cite{ono1963}.
According to this paper,
\[ \vol(T(\AD_F)^1/T(F))=   \frac{\Card{\mathrm H^1(\Gamma,\bar M)}}{\Card{\Sha(T)}} .\]
Moreover, Voskresenskii  has shown (\cite{voskresenskii1970}, Theorem~6) that
\[ \Card{A(T)}\,\Card{\Sha(T)}=\Card {\mathrm H^1(\Gamma,\Pic(\bar X))}. \]
The lemma follows.
\end{proof}

We summarize our results:

\begin{theo}\label{theo:proj-main}
Let $X$ be a smooth projective toric variety over a number field~$F$.
Endow the canonical line bundle~$K_X$ with a smooth adelic
metric and let $H$ be the corresponding height function.
Then the anticanonical height zeta function  defined by
\[ Z(s)=\sum_{x\in T(F)} H_{K_X}(x)^{s} \]
is holomorphic for $\Re(s)>1$, has a meromorphic continuation
to some half-plane $\Re(s)>1-\delta$ for some positive real number~$\delta$,
with a single pole of order~$b=\rang(\Pic(X))$ at $s=1$ and has at most polynomial
growth in vertical strips. Moreover, 
\[ \Theta =  \lim_{s\ra 1} (s-1)^b Z(s)
 = \Card{\mathrm H^1(\Gamma,\Pic (\bar X ))}\,
\mathscr X_{\Lambda_\eff(X)}(\rho)  \, \tau_X(X(\AD_F)^{\Br(X)}). \]
\end{theo}

\begin{coro}
Under the hypothesis of Theorem~\ref{theo:proj-main}, there exists
a monic polynomial~$P$ of degree~$b-1$ and a positive real number~$\eps$
such that one has
\[ \card\{x\in T(F)\,;\, H(x)\leq B\} = \frac\Theta{(b-1)!} \, B P(\log B) + \mathrm O(B^{1-\eps}),  \qquad  B\ra\infty. \]
Moreover, rational points of height~$\leq B$ equidistribute
towards the probability measure
\[ \frac 1{\tau_X(X(\AD_F)^\Br)} \, \tau_X|_{X(\AD_F)^{\Br(X)}} \]
on $X(\AD_F)^{\Br(X)}$.
\end{coro}
Since the Tamagawa measure has full support on~$X(\AD_F)$,
the equidistribution statement gives a quantitative refinement
to the density of $T(F)$ in $X(\AD_F)^{\Br(X)}$
established in Lemma~\ref{lemm.brauer}.

\begin{proof}
The first statement follows the Theorem using a standard
Tauberian theorem (see, \eg, Theorem~\ref{theo.tauber}
of~\cite{chambert-loir-tschinkel2010}).
As already observed by Peyre in~\cite{peyre95}
(see also the abstract equidistribution theorem established
in Proposition~\ref{prop.abstract.equi} 
of~\cite{chambert-loir-tschinkel2010}), 
the second one is then a consequence of this formula,
applied to any smooth adelic metric, combined with the fact that smooth
adelic metrics are dense in the space of continuous adelic metrics.
\end{proof}

\subsection{Leading term and equidistribution for integral points}

We return to the case of integral points.
Recall that
\begin{equation}\label{eqn:theta} 
\lim_{t\ra 1} (t-1)^{b}\mathrm Z(t\rho)
=  
     \sum_{A\in\Clanmax_\infty(D) }
\Theta_A
\end{equation}
where, for any $A\in\Clanmax_\infty(D)$, $\Theta_A$ is given by the formula
\begin{multline*}\label{eqn:thetaA-limit} \Theta_A  = 
 \Card{A(T,U,K)^*}\, c_T \, \mathscr X_{\Lambda'_A}(\pi(\tilde\lambda))  \times \\
 \quad \times \left(\lim_{t\ra 1}(t-1)^{b+\rang M} \int_{T(\AD_F)^{A(T,U,K)^*}}
  H( x;\rho+(t-1)\lambda ) \delta( x)\theta_A( x)\,
   \mathrm d\mu_T( x)\right) .
\end{multline*}

Let us choose a family of representatives $(\xi_1,\dots,\xi_m)$ for the
finitely many cosets of~$T(\AD_F)^{A(T,U,K)^*}$ in~$T(\AD_F)$.
For any $x\in T(\AD_F)$, let $j(x)$ be the unique integer
such that $x\in \xi_{j(x)} T(\AD_F)^{A(T,U,K)^*}$.
We define a function~$\delta^\perp$ on~$T(\AD_F)$ by
\[ \delta^\perp(x) = \begin{cases} \delta(x) & \text{if $x\in T(\AD_F)^{A(T,U,K)^*}$; } \\
0 & \text{otherwise}. \end{cases} \]
There exists a finite set~$S$ of finite places of~$F$ such that
$\xi_{j,v}\in K_v$ for $v\not\in S$. 
Consequently, there exists a continuous function~$\delta^\perp_S$
on $\prod_{v\in S} T(F_v)$ such that 
\[ \delta^\perp (x)  =\delta^\perp_S( x) \prod_{v\not\in S} \delta_v(x_v). \]

We now apply limit formulae from
of our paper~\cite{chambert-loir-tschinkel2010}. 
Indeed, a straightforward generalization of Equation~\eqref{eqn.Z*(1)-adelic} 
for the ring of finite adeles~$\AD_F^S$,
Proposition~\ref{prop.igusa-volume} for archimedean places,
as well as the convergence statement of that proposition
for places in~$S$, we obtain:
\begin{align*}
\lim_{t\ra 1}(t-1)^{b+\rang M} \int_{T(\AD_F)^{A(T,U,K)^*}}
  H(x;\rho+(t-1)\lambda ) \delta( x)\theta_A( x)\,
   \mathrm d\mu_T( x) \hskip -11cm \\
& = \lim_{t\ra 1} (t-1)^{b+\rang M} 
\int_{T(\AD_F)}
  H(x;\rho+(t-1)\lambda ) \tilde \delta( x)\theta_A( x)\,
   \mathrm d\mu_T( x)  \\
&= 
\int\prod_{v\in S}H_v(x;\rho)  \delta^\perp_S(x)   \prod_{v\not\in S}\delta_v(x_v) 
\theta_A(x)
\prod_{v\in S}\mathrm d\mu_{T}(x_v)\, \mathrm d\tau_U^S (x^S) 
\prod_{v\mid\infty} \mathrm d\tau_{A_v}(x_v),
\end{align*}
where the last integral is over 
\[ {\prod_{v\in S}T(F_v)\, U(\AD_F^S) \, \prod\limits_{v\mid\infty} D_{A_v}(F_v)} . \]
Recall that $\theta_{A}\equiv 1$ on~$\prod_{v\mid\infty}D_{A_v}(F_v)$.
By definition of the measures $\tau_{U}^S$ on $U(\AD_F^S)$
and~$\tau_U^\fin$ on $U(\AD_F^\fin)$, we thus have
\begin{multline*}
\lim_{t\ra 1}(t-1)^{b+\rang M} \int_{T(\AD_F)^{A(T,U,K)^*}}
  H(x;\rho+(t-1)\lambda ) \delta( x)\theta_A( x)\,
   \mathrm d\mu_T( x) \\
 = 
\int_{ U(\AD_F^S) \, \prod\limits_{v\mid\infty} D_{A_v}(F_v)} 
\delta^\perp_S(x)   \prod_{v\not\in S}\delta_v(x_v) 
\, \mathrm d\tau_U^\fin (x^\fin) 
\prod_{v\mid\infty} \mathrm d\tau_{A_v}(x_v).
\end{multline*}

 By Lemma~\ref{lemm.vosk},
 \[ \Card{A(T)}c_T  = \frac
 {\Card {\mathrm H^1(\Gamma,\Pic(\bar X))} }
 {\Card {\mathrm H^1(\Gamma,\bar M)} } \]
 so that,
 \begin{multline}\label{eqn:thetaA}
  \Theta_A = 
 \frac{\Card{A(T,U,K)^*}}{\Card{A(T)^*}}
 \frac{\Card {\mathrm H^1(\Gamma,\Pic(\bar X))} }
       {\Card {\mathrm H^1(\Gamma,\bar M)} } 
 \mathscr X_{\Lambda'_A}(\rho) \times \\
 \times
\int_{ U(\AD_F^S) \, \prod\limits_{v\mid\infty} D_{A_v}(F_v)} 
\delta^\perp_S(x)   \prod_{v\not\in S}\delta_v(x_v) 
\, \mathrm d\tau_U^\fin (x^\fin) 
\prod_{v\mid\infty} \mathrm d\tau_{A_v}(x_v).
 \end{multline}
 
In the case of rational points, 
the positivity of the constant~$\Theta_\emptyset$ was a straightforward consequence 
of the fact that $X(\AD_F)^{\Br(X)}$
is open and non-empty in~$X(\AD_F)$.
For integral points, an additional argument
is necessary. Indeed, the group $A(T,U,K)^*$ may impose additional
non-trivial obstructions  to the \emph{existence} of integral points.


\begin{defi}
One says that $\mathscr U$ possesses an automorphic obstruction
to the existence of integral points if there does not exist
a point $(x_v)\in T(\AD_F)^{\perp}$ such that $x_v\in\mathscr U(\mathfrak o_v)$
for all finite places~$v$.
\end{defi}

\begin{lemm}\label{lemm:non-vanishing}
Assume that $\mathscr U$ does not possess an automorphic obstruction
to the existence of integral points.
Then, for any $A\in\Clanmax_\infty(D)$, one has $\Theta_A>0$.
\end{lemm}
\begin{proof}
Let $\xi =(\xi_v)$ be any point in~$T(\AD_F)^\perp$
such that $\xi_v\in \mathscr U(\mathfrak o_v)$ for all finite  places~$v$.
Since all characters in $A(T,U)^*$ are trivial on~$T(F_\infty)$,
we may assume that $\xi_v\in A_v$ for all archimedean places~$v$.
The function
\[ x\mapsto \delta^\perp_S(x) \prod_{v\not\in S} \delta_v(x_v) \]
under the integral sign
in the above limit  formula 
is continuous and positive at this point~$\xi$.
Since it belongs to the support of
the measure~$\mathrm d\tau_U^\fin \prod_{v\mid\infty}\mathrm d\tau_{A_v}$,
we obtain that $\Theta_A>0$, as claimed.
\end{proof}

Applying the tauberian theorem~\ref{theo.tauber.S}
of~\cite{chambert-loir-tschinkel2010},
this concludes the proof of  
our main result concerning the number of
integral points of bounded height:
\begin{theo}
\label{thm:integral-main}
Let $X$ be a smooth projective toric variety over a number field~$F$.
Let $D$ be an invariant divisor such that $-(K_X+D)$ is big
and let $U=X\setminus D$. Let $\mathscr U$ be a model
of~$U$ over~$\mathfrak o_F$. Assume that $\mathscr U$
does not possess an automorphic obstruction to the existence
of integral points.
Let us endow the log-anticanonical line bundle~$-(K_X+D)$ 
with a smooth adelic metric; let $H$ be the corresponding height function.
Then, when $B\ra\infty$,
\[ \card \{x\in T(F)\cap\mathscr U(\mathfrak o_F)\,;\, H(x)\leq B\} 
= \frac\Theta{(b-1)!} B (\log B)^{b-1} (1+ \mathrm O(1/\log B)), \]
where 
\[ b= \rang \Pic(U) + \sum_{v\mid\infty} \Card{A_v}
= \rang\Pic(U)+\sum_{v\mid\infty} (1+\dim\Clan_v(D))  \]
and $\Theta$ is a positive real number given by 
\begin{multline*} \Theta 
=  \frac{\Card{A(T,U,K)^*}}{\Card{A(T)^*}}\,
  \frac
     {\Card {\mathrm H^1(\Gamma,\Pic(X_E))} }
     {\Card {\mathrm H^1(\Gamma,M_E)} }  \times \\
\times
     \sum_{A\in\Clanmax_\infty(D) }
\mathscr X_{\Lambda'_A}(\rho)
\int_{U(\AD_{F,\fin})\prod D_{A_v}(F_v)}
\delta^\perp(x) 
\mathrm d\tau_U^\fin (x_\fin)
\prod_{v\mid \infty}\mathrm d\tau_{A_v}(x_v). \end{multline*}
\end{theo}


Theorem~\ref{thm:integral-main} holds for any smooth metrization
of the log-anticanonical line bundle. By the abstract
equidistribution theorem of~\cite{chambert-loir-tschinkel2010}
(Prop.~\ref{prop.abstract.equi}), we obtain the following corollary:
\begin{coro}
Retain the hypotheses of Theorem~\ref{thm:integral-main}.
Then, when $B\ra \infty$,
the integral points of height~$\leq B$ equidistribute
towards the unique probability measure 
on $T(\AD_F)^\perp\cap U(\AD_{F,\fin})\times \prod_{v\mid\infty}{D}_{A_v}(F_v)$
which is proportional to
\[ 
     \sum_{A\in\Clanmax_\infty(D) }
\mathscr X_{\Lambda'_A}(\rho)
\int_{U(\AD_{F,\fin})\prod D_{A_v}(F_v)}
\mathbf 1_{X(\AD_F)^\Br}
\delta^\perp(x) 
\mathrm d\tau_U^\fin (x_\fin)
\prod_{v\mid \infty}\mathrm d\tau_{A_v}(x_v). \]
\end{coro}

\begin{rema}
Langlands has parametrized the group of
automorphic characters of an algebraic torus in terms
of Galois cohomology of the dual L-group.
This has been generalized by Kottwitz--Shelstad and Nyssen
(see~\cite{nyssen1995})
to complexes of tori.
Comparing this description with the outcome
of the Hochschild--Serre spectral sequence
for the étale cohomology of~$\gm$
on~$U_{\bar F}$ strongly suggests that the kernel of $A(T,U)^*$
in $T(\AD_F)$ is cut out by the Brauer--Manin obstruction
defined by the algebraic part of~$\Br(U)$. 
This holds indeed when $U$ is proper (see~\cite{sansuc1981}
and Lemma~\ref{lemm.brauer} above) 
or when $U=T$ (see~\cite{harari2010}).
\end{rema}